\documentclass{article}
\usepackage{authblk}
\usepackage{amsfonts,amsmath, amssymb,mathrsfs,mathtools,amsthm} 
\usepackage{hyperref, enumitem}

\usepackage[doi=false, url=false, eprint=false, maxbibnames=10, firstinits=true]{biblatex}
\newtheorem{thm}{Theorem}[section]
\newtheorem{prop}[thm]{Proposition}
\newtheorem{lem}[thm]{Lemma}
\newtheorem{rem}[thm]{Remark}
\newtheorem{define}[thm]{Definition}
\numberwithin{equation}{section}

\newcommand{\R}{\mathbb{R}}
\newcommand{\supp}{\operatorname{supp}}
\newcommand{\A}{\mathcal{A}}
\newcommand{\F}{\mathscr{F}}
\newcommand{\LRFH}{{\rm LRFH}}
\newcommand{\Crit}{\operatorname{Crit}}
\newcommand{\wt}{\widetilde}

\title{Two-boost problem\\ for the rotating Kepler problem}
\author{Jagna Wi\'{s}niewska }
\affil{Department of Mathematics, University of Augsburg}

\begin{document}

\maketitle

\begin{abstract}
The two-boost problem in space mission design asks whether two
points of phase space can be connected with the help of two boosts of given
energy. In this article we provide a positive answer for the rotating Kepler problem by generalizing the definition of Lagrangian Rabinowitz Floer homology from the class introduced in \cite{Wisniewska2024} to a strictly broader class, and computing the corresponding homology in this more general setting. The principal technical challenge is the non-compactness of the energy hypersurface both near the collision singularity and in the infinity.
\end{abstract}

\section{Introduction}

The two-boost problem considers a very practical question in astrodynamics and the space mission design: Can we send a rocket in the gravitational field of one, or more, celestial bodies, using the engines only at the beginning and at the end of the journey?

The research on the two-boost problem has been initiated by the ground-breaking work of Hohmann on the attainability of heavenly bodies \cite{Hohmann}. Hohmann considered the two-boost problem in the setting of the Kepler problem, with the rocket moving under the influence of the gravitational field of just one heavenly body. The Hohmann transfer is a transfer between two circular orbits in the Kepler problem with the help of a Kepler ellipse which is tangent to the two circles. It requires two tangential boosts, one to transfer from the first circle to the ellipse and a second one to transfer from the ellipse to the second circle. This means that for the Kepler problem, two points in the phase space can always be connected with the help of two boosts. Although this work was written almost a hundred years ago, the Hohmann transfer is still one of the crucial ingredients in space mission design \cite{Vallado}.

In this paper we answer the two-boost problem for the rotating Kepler problem.
The rotating Kepler problem models the motion of a massless body in the gravitational field of a (heavy) celestial body in the rotating coordinates. It differs from the stationary Kepler problem by the presence of the Coriolis force. The corresponding Hamiltonian $K: T^*(\R^2\setminus\{0\})\to \R$ is of the form:
\begin{equation}\label{Kepler}
K(q,p) := \frac{1}{2}|p|^2 + p_1q_2-p_2q_1 - \frac{1}{|q|}.
\end{equation}
Note that all the level sets of $K$ are non-compact. In fact, they all have two types of non-compactness: one in the infinity representing the possibility of the satellite flying away into cosmos, and the other one coming from the singularity and modelling the scenario of the satellite crushing into the celestial body - a collision. 

The general setup for the two-boost problem is as follows: Consider the cotangent bundle $T^*Q$ of a manifold $Q$ with its canonical exact symplectic form $\omega=d\lambda$, $\lambda=p\,dq$, and a Hamiltonian $H:T^*Q\to\R$. Given two points $q_0,q_1\in Q$ and an energy value $c$, the two-boost problem asks for the existence of a Hamiltonian orbit of energy $c$ connecting the cotangent fibers $T_{q_0}^*Q$ and $T_{q_1}^*Q$. In other words, we are interested in chords $(v,\eta)\in W^{1,2}([0,1], T^*Q)\times\R$ satisfying
\begin{equation}\label{Def(v,eta)}
\partial_tv=\eta X_H(v),\quad v([0,1])\subseteq H^{-1}(c),\quad \textrm{and}\quad v(i)\in T_{q_i}^*Q\quad \textrm{for}\quad i=0,1.
\end{equation}

In the following theorem we answer the two-boost problem for the rotating Kepler problem:
\begin{thm}\label{thm:2Boost}
Let $K$ be the Hamiltonian defined in \eqref{Kepler}.
Fix $q_0,q_1\in \R^2\setminus \{0\}$ and denote $R:= \max\{|q_0|, |q_1|\}$. Then for every $c \geq 2 R^2$ on $K^{-1}(c)$ there exist at least two solutions 
to \eqref{Def(v,eta)} on the energy level $K^{-1}(c)$ (including collision trajectories\footnote{An exact definition of a collision trajectory can be found in section \ref{ssec:Levi-Civita}.}).
\end{thm}

To prove this theorem we first formulate the two-boost problem as an variational problem
using the \emph{Rabinowitz action functional}, whose critical points correspond to the solutions of \eqref{Def(v,eta)}. Let us recall its definition: 

For $q_0,q_1\in Q$ we consider the following space of paths between $q_0$ and $q_1$:
$$
\mathscr{H}_{q_0,q_1}:= \left\lbrace v \in W^{1,2}([0,1], T^*Q)\ \big|\ v(i)\in T^*_{q_i}Q\quad \textrm{for}\quad i=0,1\right\rbrace.
$$ 
For a Hamiltonian $H:T^*Q \to \R$ and energy value $c\in \R$ the corresponding \emph{Rabinowitz action functional} is defined
\begin{align*}
\A^{H-c}_{q_0,q_1} & :\mathscr{H}_{q_0,q_1}\times \R\to\R,\\
\A^{H-c}_{q_0,q_1}(v,\eta) & := \int_0^1 \lambda(\partial_t v)dt - \eta \int_0^1 H(v(t))dt.
\end{align*}
This action functional was introduced by Rabinowitz in \cite{Rabinowitz1978} on the loop space to study the existence of periodic orbits on compact star-shaped hypersurfaces in $\R^{2n}$. His findings inspired Weinstein \cite{Weinstein1979} to formulate the famous \emph{Weinstein conjecture} on the existence of periodic orbits on compact contact-type hypersurfaces. In \cite{Floer1989} Floer revolutionized the field of symplectic geometry by constructing a homology using the symplectic action functional. The chain complex of this homology, called \emph{Floer homology}, is generated by the critical points of the symplectic action functional, which correspond to the $1$-periodic orbits of the Hamiltonian flow. The great advantage of using the Floer theory in the analysis of Hamiltonian systems is the continuation principle, which assures that a Floer-type homology is invariant under compact perturbations of the Hamiltonian. In particular, under the perturbations of the system the set of solutions may undergo bifurcations, yet the the associated Floer-type homology will remain unchanged as long as the geometric properties of the hypersurface are preserved.

In \cite{CieliebakFrauenfelder2009} Cieliebak and Frauenfelder considered the 
Rabinowitz action functional on the free loop space associated to a compact, contact-type hypersurface in an exact symplectic manifold. They constructed a Floer-type homology for this setting and called it the Rabinowitz Floer homology. However, for the purpose of studying the two-boost problem we are more interested in the Floer-type homology of the Rabinowitz action functional defined on the chord space between two Lagrangians. This algebraic invariant was constructed by Merry in \cite{Merry2014} and is called the Lagrangian Rabinowitz Floer homology. He showed that for high enough energy level sets in magnetic Hamiltonian systems the corresponding Lagrangian Rabinowitz Floer homology (LRFH) is well defined and non-trivial, hence providing a positive answer to the two-boost problem for this setting. However, in both cases the Floer-type homologies were defined under the assumption that the energy hypersurface $H^{-1}(c)$ is \emph{compact}. Only recently, this author together with Pasquotto generalized the construction of Rabinowitz Floer homology to include non-compact hypersurfaces \cite{Pasquotto2017, Pasquotto2018}.

The first step of extending the definition of Lagrangian Rabinowitz Floer homology to \emph{non-compact} hypersurfaces has been done by this author together with Cieliebak, Frauenfelder and Miranda in \cite{Wisniewska2024}. In this paper we build on this result and extend it to a larger class of hypersurfaces in order to apply it in the proof of Theorem \ref{thm:2Boost}.

For a fixed constant $a>0$ we define a Hamiltonian on $T^*\R^2$:
\begin{equation}\label{DefHa}
H_a:=\frac{1}{2}(p_1^2+p_2^2)+a(p_1q_2-p_2q_1)=\frac{1}{2}\left( p_r^2+\frac{p_\theta^2}{r^2}\right)+ap_\theta.
\end{equation}

Furthermore, we define a set of perturbations $\mathcal{H}\subseteq C^\infty(T^*\R^2)$, such that $h\in \mathcal{H}$ if and only if $dh$ is compactly supported and
\begin{equation}\label{DefH}
h(q,p)-dh(p\partial_p)+\frac{1}{2}|p|^2>0, \qquad \forall\ (q,p)\in T^*\R^2.
\end{equation}
Let us denote the set of such Hamiltonians $\mathcal{H}$. Note, that $\mathcal{H}$ is convex. Moreover, the set of positive constants satisfies $\R_+\subseteq \mathcal{H}$ and for every $a\in \R_+$ and $h\in \mathcal{H}$ we have $h+a\in \mathcal{H}$.

In the following theorem we define the Lagrangian Rabinowitz Floer ho\-mo\-lo\-gy in this setting and show that it is invariant under the set of perturbations $\mathcal{H}$ as in \eqref{DefH}:
\begin{thm}\label{thm:main}
For any Hamiltonian $H_a$ as in \eqref{DefHa}, a perturbation $h \in \mathcal{H}$, and any two points $q_0,q_1\in \R^2$, the Lagrangian Rabinowitz Floer homology $LRFH_*\left(\A^{H_a-h}_{q_0,q_1}\right)$ is well defined. Moreover, for every two perturbations $h_0,h_1 \in \mathcal{H}$ the associated homologies are isomorphic:
$$
\LRFH_*\left(\A^{H_a-h_0}_{q_0,q_1}\right)=\LRFH_*\left(\A^{H_a-h_1}_{q_0,q_1}\right).
$$
\end{thm}
Having defined the Lagrangian Rabinowitz Floer homology we show that it is non-vanishing, as that will imply that the critical set of the associated Rabinowitz action functional is non-empty, which is equivalent with the existence of solutions to the two-boost problem \eqref{Def(v,eta)}. However, to prove the existence of solutions to \eqref{Def(v,eta)} it suffices to show that the positive Lagrangian Rabinowitz Floer homology is non-vanishing. That is because the complex of the positive Lagrangian Rabinowitz Floer homology is generated by the solutions $(v,\eta)$ of \eqref{Def(v,eta)} with $\eta>0$. We calculate the positive Lagrangian Rabinowitz Floer homology in the following theorem:

\begin{thm}\label{thm:posLRFH}
For any Hamiltonian $H_a$ as in \eqref{DefHa}, a perturbation $h \in \mathcal{H}$, and any two points $q_0,q_1\in \R^2$, the positive Lagrangian Rabinowitz Floer homology $LRFH_*^+(\A^{H_a-h}_{q_0,q_1})$ is well defined and equal to
$$
\LRFH_*^+\left(\A^{H_a-h}_{q_0,q_1}\right) = 
\begin{cases}
\mathbb{Z}_2 & \textrm{for}\quad *=\frac{1}{2},\\
0 & \textrm{otherwise}.
\end{cases}
$$
\end{thm}
An immediate consequence of this theorem is that $\Crit \left(\A^{H_a-h}_{q_0,q_1}\right)$ is non-empty for all Hamiltonians $H_a$ as in \eqref{DefHa} and $h \in \mathcal{H}$, which provides an affirmative answer to the two-boost problem in this case.
However, our goal is to prove Theorem \ref{thm:2Boost} and to answer the two-boost problem for energy level sets of Hamiltonian $K$ as in \eqref{Kepler}. Thus, we need to relate the two Hamiltonians $K$ and $H_a-h$. 

First, recall that the the Hamiltonian $K$ of the rotating Kepler problem introduced at \eqref{Kepler} is not defined on the whole $T^*\R^2$, but has a singularity at $0$. The singularity models the possibility of the collision of the spacecraft with the celestial body. In order to deal with the singularity of the Kepler problem and the corresponding non-compactness of the energy level sets, we will use a method presented in \cite{Levi1920} by Levi-Civita called the \emph{Levi-Civita regularization}.

For a fixed $c>0$ Levi-Civita defined the following Hamiltonian on $T^*\R^2$
\begin{equation}\label{regHam}
K_c(q,p):=\frac{1}{2}|p|^2+2 |q|^2(p_1q_2-p_2q_1-2c)-4.
\end{equation}
We will refer to it as the Hamiltonian of the regularized rotating Kepler problem or simply as the regularized Hamiltonian. The Hamiltonian dynamics on $K_c^{-1}(0)$ relates to the Hamiltonian dynamics on $K^{-1}(c)$ in the following way:

\begin{rem}\label{rem:DoubleCover}
Fix a regular value $c\in \R$ of $K$ and two points $z_0,z_1\in \mathbb{C}\setminus\{0\}$. Then for every solution $(v, \eta)\in \mathscr{H}_{z_0^2,z_1^2}\times \R$ of \eqref{Def(v,eta)} on $K^{-1}(c)$ there exist two solutions $(v_1,\eta_1)\in \mathscr{H}_{z_0,z_1}\times \R$ and $(v_2,\eta_2)\in \mathscr{H}_{-z_0,-z_1}\times \R$ of \eqref{Def(v,eta)} on $K_c^{-1}(0)$. And vice versa: for every solution $(v,\eta)\in \mathscr{H}_{z_0,z_1}\times \R$ of \eqref{Def(v,eta)} on $K_c^{-1}(0)$ satisfying $v([0,1])\cap T_0^*\mathbb{C}=\emptyset$, there exists a solution $(\tilde{v},\tilde{\eta}) \in \mathscr{H}_{z_0^2,z_1^2}\times \R$ of \eqref{Def(v,eta)} on $K^{-1}(c)$.
If a solution $(v,\eta)\in \mathscr{H}_{z_0,z_1}\times \R$ of \eqref{Def(v,eta)} on $K_c^{-1}(0)$ satisfies $v([0,1])\cap T_0^*\mathbb{C}\neq \emptyset$, then there exists a collision trajectory $(\tilde{v},\tilde{\eta})$ which solves \eqref{Def(v,eta)} on $K^{-1}(c)$, everywhere outside of $T_0^*\mathbb{C}$.
\end{rem}

In the following Theorem we show how the dynamics of the Hamiltonians $K_c$ as in \eqref{regHam} and $H_a-h$, with $H_a$ as in \eqref{DefHa} and $h$ as in \eqref{DefH}, are related to each other:

\begin{thm}\label{thm:Kepler}
Fix $q_0,q_1\in \R^2$ and denote $R:= \max\{|q_0|, |q_1|\}$. Then for every $c \geq \sqrt{2}R$, there exists a Hamiltonian $H_a$ as in \eqref{DefHa} with $a> 2 R^2$
and $h\in \mathcal{H}$ as in \eqref{DefH}, such that
$$
\Crit\A_{q_0,q_1}^{H_a-h}\subseteq \Crit \A_{q_0,q_1}^{K_c}.
$$
\end{thm}

We will now show how Theorem \ref{thm:2Boost} is proven by combining the results from Theorem \ref{thm:main}, Theorem \ref{thm:posLRFH}, Remark \ref{rem:DoubleCover} and Theorem \ref{thm:Kepler}. Let $q_0,q_1\in \mathbb{C}\setminus \{0\}$ be two points and let $c \geq \sqrt{2} \max\{|q_0|, |q_1| \}$. 
On one hand, by Remark \ref{rem:DoubleCover} we know that every element of $\Crit \A_{\pm q_0,\pm q_1}^{K_c}$ can be projected onto a solution in $\Crit \A_{q_0^2,q_1^2}^{K-c}$ (or a collision trajectory). Since $K_c^{-1}(0)\setminus T^*_0\mathbb{C}$ is a double cover of $K^{-1}(c)$, we know in fact that for every solution $(v_1,\eta_1)\in \Crit \A_{q_0,q_1}^{K_c}$ there exists a solution $(v_2, \eta_2)\in \Crit \A_{-q_0,-q_1}^{K_c}$, such that $(v_1,\eta_1)$ and $(v_2,\eta_2)$ correspond to a single solution in $\Crit \A_{q_0^2,q_1^2}^{K-c}$, possibly a collision trajectory. The same reasoning applies to the solutions from $\Crit \A_{-q_0,q_1}^{K_c}$ and $\Crit \A_{q_0,-q_1}^{K_c}$. 

On the other hand, by Theorem \ref{thm:posLRFH} and Theorem \ref{thm:Kepler} we have:
$$
\Crit \A_{q_0,q_1}^{K_c} \neq \emptyset, \quad \Crit \A_{-q_0,q_1}^{K_c} \neq \emptyset, \quad \Crit \A_{q_0,-q_1}^{K_c}\neq \emptyset \quad \textrm{and}\quad \Crit \A_{-q_0,-q_1}^{K_c}\neq \emptyset.
$$ 
Therefore the cardinality of $\Crit \A_{q_0^2,q_1^2}^{K-c}$ together with the set of collision trajectories from $q_0^2$ to $q_1^2$ is equal to the sum of $\# \Crit \A_{q_0,q_1}^{K_c}+\# \Crit \A_{-q_0,q_1}^{K_c}\geq 2$, which concludes the proof of Theorem \ref{thm:2Boost}.






\subsection*{Structure of the paper}

The paper is organized as follows:
In subsections \ref{ssec:Rab}, \ref{ssec:Floer} and \ref{ssec:LRFH} we recall the setting and the definition of the Lagrangian Rabinowitz Floer homology. In subsection \ref{ssec:LRFH} we cite \cite[Thm. 2.5]{Wisniewska2024}, which states the conditions under which the Lagrangian Rabinowitz homology is well defined for non-compact hypersurfaces. The proof that Hamiltonians of the form $H_a-h$ with $H_a$ is as in \eqref{DefHa} and $h\in \mathcal{H}$ as in \eqref{DefH} satisfy conditions required in \cite[Thm. 2.5]{Wisniewska2024} occupies most of Section \ref{sec:Bounds}. The most challenging of those conditions is the proof of the boundedness of Floer trajectories. It relies on the methodology presented in the proof of \cite[Thm. 3.3]{Wisniewska2024}. In Section \ref{sec:posLRFH} we show that the positive Lagrangian Rabinowitz Floer homology of Hamiltonians $H_a-h$ is also well defined and we relate it to the homology calculated in \cite[Thm. 1.2]{Wisniewska2024}, which together proves Theorem \ref{thm:posLRFH}.

In subsection \ref{ssec:reg} we explain how the dynamics of the regularized Hamiltonian $K_c$ as in \eqref{regHam} describes the dynamics of the rotating Kepler problem induced by $K$ as in \eqref{Kepler}. We also discuss the collision trajectories in the rotating Kepler problem. In Section \ref{sec:RegHam} we show how we can construct a Hamiltonian of the form $H_a-h$, which dynamics on a compact subset of $(H_a-h)^{-1}(0)$ would correspond to the dynamics on $K_c^{-1}(0)$. Finally, in Section \ref{sec:RegHam} we prove Theorem \ref{thm:Kepler}.

\subsection*{Acknowledgements}

I would like to thank prof. Urs Frauenfelder for all the time devoted to mentoring me and all the fruitful discussions we had on the subject. I would also like to express my gratitude to Hanna H\"{a}u\ss ler, Simone Zahn and Milan Zerbin for all the ice-cream breaks and after-work hangouts without which the creation of this article would not have been possible. Last, but not least I would like to thank Jennifer Gruber whose professionalism at providing administrative support at Augsburg University made it possible for me to focus on my research.

My postdoctoral position at University of Augsburg is funded by the Deutsche
Forschungsgemeinschaft (DFG, German Research Foundation) via the grant “Himmelsmechanik, Hydrodynamik und Turing-Maschinen” - 541525489.

\section{Setting}

\subsection{Rabinowitz action functional}
\label{ssec:Rab}

Throughout this section, $Q$ denotes a smooth oriented $n$-dimensional manifold, and $T^*Q$ its cotangent bundle equipped with the exact symplectic form $\omega=d\lambda$ for the canonical $1$-form $\lambda:=p\,dq$. 
Fix $q_0,q_1 \in Q$ and define the space of paths
$$
\mathscr{H}_{q_0,q_1}:= \left\lbrace v \in W^{1,2}([0,1], T^*Q)\ \big|\ v(i)\in T^*_{q_i}Q\quad \textrm{for}\quad i=0,1\right\rbrace.
$$
Consider a Hamiltonian $H: T^*Q \to \R$ with regular level set $H^{-1}(0)$. The following action functional has been first introduced by Rabinowitz in \cite{Rabinowitz1978} in order to prove the Weinstein conjecture for star-shaped hypersurfaces. However, in the original setting the domain of the functional was the loop space, whereas here we define the functional $\A^H_{q_0,q_1}:\mathscr{H}_{q_0,q_1}\times \R\to\R$ associated to $H$ as
$$
\A^H_{q_0,q_1}(v,\eta):= \int_0^1 \lambda(\partial_t v)dt - \eta \int_0^1 H(v(t))dt.
$$
We call it the {\em Rabinowitz action functional}.

The derivative of $\A^H_{q_0,q_1}$ in direction $(\xi, \sigma) \in T_v \mathscr{H}_{q_0,q_1}\times T_\eta \R$ equals
\begin{equation}\label{eq:dAH}
d\A^H_{q_0,q_1}(v,\eta)[\xi,\sigma] = \int_0^1 \omega (\xi,\partial_tv -\eta X_H) - \sigma\int_0^1 H(v)dt.
\end{equation}
Consequently, $(v,\eta) \in \Crit (\A^H_{q_0,q_1})$ if and only if it satisfies \eqref{Def(v,eta)}.
Thus we can have three types of critical points: 
\begin{itemize}
\item $\eta>0$ and $\wt v(t):=v(t/\eta)$ is a Hamiltonian chord (i.e., an integral curve of $X_H$) on $H^{-1}(0)$ from $T^*_{q_0}Q$ to $T_{q_1}^*Q$;
\item $\eta<0$ and $\wt v(t):=v(t/\eta)$ is a Hamiltonian chord on $H^{-1}(0)$ from $T^*_{q_1}Q$ to $T_{q_0}^*Q$; 
\item $\eta=0$ and $v$ is a constant path in $T^*_{q_0}Q\cap T_{q_1}^*Q\cap H^{-1}(0)$ (which can only occur if $q_0=q_1$). 
\end{itemize}
In particular, if $T^*_{q_0}Q \cap H^{-1}(0)=\emptyset$ or $T_{q_1}^*Q \cap H^{-1}(0)=\emptyset$, then $ \Crit (\A^H_{q_0,q_1})=\emptyset$. Therefore, from now on we will assume that 
\begin{equation}\label{T*QcapHnonempty}
T^*_{q_0}Q \cap H^{-1}(0)\neq\emptyset,\quad T^*_{q_1}Q \cap H^{-1}(0)\neq\emptyset,\quad\textrm{and } T^*_{q_0}Q\pitchfork H^{-1}(0) \text{ if }q_0=q_1.
\end{equation}

In \cite{Merry2014} Merry defined the Lagrangian Rabinowitz Floer homology as an algebraic invariant of a pair of Lagrangians $L_0,L_1$ and a compact hypersurface $\Sigma$ of exact contact type in an exact symplectic manifold $(M, \omega)$.
The regular level set $H^{-1}(0)$ is said to be of {\em exact contact type} if there exists a Liouville vector field $Y$ on $T^*Q$ (satisfying $L_Y\omega=\omega$) such that $dH(Y)>0$ along $H^{-1}(0)$. 

Note that for every Liouville vector field $Y$ setting $(\xi,\sigma)=(Y,0)$ in equation~\eqref{eq:dAH} we obtain the following relation:
\begin{equation}\label{dA^H(Y)}
  \A^H (v,\eta) - d\A^H(v,\eta)[Y, 0] = \eta\int_0^1 \left(dH(Y)-H\right)(v(t))dt.
\end{equation}

In order to construct Lagrangian Rabinowitz Floer homology for a non-compact hypersurface we want the critical set of the action functional to be bounded in $L^\infty$. In fact, to construct Lagrangian Rabinowitz Floer homology we will need the associated Rabinowitz action functional to have its critical set \emph{continuously compact}. The following definition has bee introduced in \cite[Def. 2.1]{Wisniewska2024}:
\begin{define}\label{def:contComp}
Consider a Hamiltonian $H: T^*Q \to \R$, such that $H^{-1}(0)$ is noncompact. Fix $q_0,q_1\in T^*Q$ satisfying \eqref{T*QcapHnonempty}. Let $K\subseteq T^*Q$ be a compact set and let 
$$
\mathcal{H}\subseteq \{ h\in C^\infty(T^*Q)\ |\ dh\in C_0^\infty(K)\},
$$
be an open neighbourhood of $0$ in $\{ h\in C^\infty(T^*Q)\ |\ dh\in C_0^\infty(K)\}$.
We say that the associated Rabinowitz action functional $\A^H_{q_0,q_1}:\mathscr{H}_{q_0,q_1}\times \R\to \R$ has critical set \emph{continuously compact} in $(K,\mathcal{H})$ if
$$
\forall\ h\in\mathcal{H}\quad \textrm{and}\quad\forall\ (v,\eta) \in \Crit \A^{H+h}_{q_0,q_1} \quad \textrm{we have}\quad v([0,1])\subseteq K.
$$
\end{define}
This property assures the boundedness of the critical set of the action functional to persist under compact perturbations of the Hamiltonian. This, in turn, is necessary to be able to perturb the Hamiltonian and assure that the associated Rabinowitz action functional is Morse.

\subsection{Floer trajectories}
\label{ssec:Floer}

In this subsection we introduce the fundamental notion for constructing any Floer-type homology - the Floer trajectories. We start by equipping the space $\mathscr{H}_{q_0,q_1}$ with a metric. An \emph{almost complex structure} $J$ on a manifold $M$ is a bundle endomorphism $J: TM \to TM$ satisfying $J^2 = -\operatorname{Id}$. An almost complex structure $J$ on a symplectic manifold $(M, \omega)$ is called \emph{compatible} if $\omega(\ \cdot\ ,J \cdot\ )$ defines a Riemannian metric on $M$. Denote by $\mathcal{J}(M,\omega)$ the space of all compatible almost complex structures on $(M,\omega)$ with the $C^\infty$-topology.
A linear algebra argument \cite[Prop. 13.1]{Silva2001} shows that $\mathcal{J}(M,\omega)$ is contractible.

A smooth $2$-parameter family $\{J_{t,\eta}\}_{(t,\eta)\in[0,1]\times \R}$ of compatible almost complex structures on $(T^*Q, \omega)$ defines an $L^2$-metric on $\mathscr{H}_{q_0,q_1}\times \R$ by
$$
\langle (\xi_1, \sigma_1), (\xi_2, \sigma_2)\rangle := \int_0^1 \omega(\xi_1(t),J_{t,\eta}(v(t))\xi_2(t) )+\sigma_1\sigma_2.
$$
for $(\xi_i,\sigma_i)\in T_{(v,\eta)}(\mathscr{H}_{q_0,q_1}\times \R)$.
The gradient of the Lagrangian Rabinowitz action functional $\A^H_{q_0,q_1}$ with respect to this metric equals
$$
\nabla \A^H_{q_0,q_1} (v,\eta)= \left( \begin{array}{c}
-J_{t,\eta}(v(t)) (\partial_tv -\eta X_H) \\ - \int H(v)dt.
\end{array}
\right)
$$
Fix an open subset $\mathcal{V}\subseteq T^*Q$ and $\mathbb{J}\in \mathcal{J}(T^*Q,\omega)$.
We denote by $\mathcal{J}(\mathcal{V}, \mathbb{J})$ the set of all smooth maps
$$
[0,1]\times \mathbb{R}\to \mathcal{J}(T^*Q,\omega),\qquad (t,\eta)\mapsto J_{t,\eta}
$$
satisfying
\begin{align*}
J_{t,\eta}(x) = \mathbb{J}(x)\text{ for } x\notin \mathcal{V}
\quad\textrm{and}\quad \sup_{(t,\eta)\in [0,1]\times \mathbb{R}}\|J_{t,\eta}\|_{C^k}<+\infty \text{ for all } k\in \mathbb{N}.
\end{align*}
A solution $u: \R \to \mathscr{H}_{q_0,q_1} \times \R$ to the gradient flow equation $\partial_s u = \nabla \A^{H_s}_{q_0,q_1}(u)$ is called a \emph{Floer trajectory}. In other words, a Floer trajectory $u=(v,\eta) \in W^{1,2}(\R\times[0,1], T^*Q)\times W^{1,2}(\R)$ is a solution to the equations
\begin{equation}\label{FloerEq}
\begin{aligned}
\partial_s v(s,t) & = - J_{s, t,\eta}(v(s,t)) (\partial_sv(s,t)-\eta(s)X_{H_s}(v(s,t))),\\
\partial_s \eta (s) & = - \int_0^1 H_s\circ v(s,t)dt,
\end{aligned}
\end{equation}
$$
v(s,0) \in T^*_{q_0}Q \quad \textrm{and} \quad v(s,1) \in T^*_{q_1}Q \quad \forall\ s\in \R.
$$

For $(x_-, x_+) \in \Crit \A^{H_-}_{q_0,q_1}\times \Crit \A^{H_+}_{q_0,q_1}$ we denote the space of Floer trajectories from $x_-$ to $x_+$ by
$$
\F_\Gamma(x_-, x_+):=
\left\lbrace\begin{array}{c|c}
& \partial_s u = \nabla \A^{H_s}_{q_0,q_1}(u),\\
\hspace*{-0.2cm}\smash{\raisebox{.5\normalbaselineskip}{ $u: \R \to \mathscr{H}_{q_0,q_1} \times \R$}}& \lim_{s\to \pm \infty}u(s)=x_\pm.
\end{array}\right\rbrace
$$
In case the homotopy $\Gamma$ is constant in $s$, i.e. $H_s\equiv H$ and $J_{s,t,\eta}\equiv J \in \mathcal{J}(\mathcal{V},\mathbb{J})$ for some Hamiltonian $H$ and a compatible almost complex structure $J$, we denote $\F_{H,J}(x_-,x_+):=\F_\Gamma(x_-, x_+)$.

Moreover, for every pair $(\mathbf{a},\mathbf{b})\in \R^2$ we denote 
\begin{equation}\label{DefM(a,b)}
\mathcal{M}^\Gamma(\mathbf{a},\mathbf{b}):=\left\lbrace\begin{array}{c|c}
& (x_-,x_+)\in \Crit\A^{H_-}_{q_0,q_1}\times \Crit\A^{H_+}_{q_0,q_1},\\
\hspace*{-0.2cm}\smash{\raisebox{.5\normalbaselineskip}{ $u\in \F_\Gamma(x_-,x_+)$}}& \A^{H_-}_{q_0,q_1}(x_-)\geq \mathbf{a}, \quad \A^{H_+}_{q_0,q_1}(x_+)\leq \mathbf{b}.
\end{array}\right\rbrace
\end{equation}
Analogously, we denote $\mathcal{M}^{H,J}(\mathbf{a},\mathbf{b}):=\mathcal{M}^\Gamma(\mathbf{a},\mathbf{b})$ whenever the homotopy $\Gamma$ is constant in $s$ and equal to the pair $(H,J)$.

If a homotopy $\Gamma$ is constant in $s$, then the action increases along Floer trajectories. However, for a nonconstant homotopy $\Gamma$, this need not be the case. To deal with this phenomenon, we introduce a condition that ensures that the action cannot decrease indefinitely along a Floer trajectory.
We say that a homotopy $\Gamma:= \{(H_s, J_s)\}_{s\in \R}$ with $H_{\pm}:=\lim_{s\to \pm \infty}H_s$ satisfies the \emph{Novikov finiteness condition} if for all $a,b\in \R$ we have
\begin{equation}\label{Novikov}
\begin{aligned}
\inf\left\lbrace\begin{array}{c|c}
& y \in  \Crit\A^{H_+}_{q_0,q_1}, \quad \exists\ x\in \Crit  \A^{H_-}_{q_0,q_1},\\
\hspace*{-0.2cm}\smash{\raisebox{.5\normalbaselineskip}{ $\A^{H_+}_{q_0,q_1}(y)$}}& \A^{H_-}_{q_0,q_1}(x)\geq a, \quad \F_\Gamma(x,y)\neq \emptyset.
\end{array}\right\rbrace & > -\infty,\\
\sup\left\lbrace\begin{array}{c|c}
& x \in  \Crit\A^{H_-}_{q_0,q_1}, \quad \exists\ y\in \Crit  \A^{H_+}_{q_0,q_1},\\
\hspace*{-0.2cm}\smash{\raisebox{.5\normalbaselineskip}{ $\A^{H_-}_{q_0,q_1}(x)$}}& \A^{H_+}_{q_0,q_1}(y)\leq b, \quad \F_\Gamma(x,y)\neq \emptyset.
\end{array}\right\rbrace &< +\infty.
\end{aligned}
\end{equation}

\subsection{Lagrangian Rabinowitz Floer homology}
\label{ssec:LRFH}

The Lagrangian Rabinowitz Floer homology as an algebraic invariant of a pair of Lagrangians $L_0,L_1$ and a compact hypersurface $\Sigma$ of exact contact type in an exact symplectic manifold $(M, \omega)$ was defined by Merry in \cite{Merry2014}. In \cite{Wisniewska2024} Cieliebak, Frauenfelder, Miranda and the author of this paper extended the definition of Lagrangian Rabinowitz Floer homology to include non-compact hypersurfaces. In particular, we have proven the following theorem \cite[Thm. 2.5]{Wisniewska2024} that states the conditions sufficient to define the Lagrangian Rabinowitz Floer homology in the non-compact setting: 

\begin{thm}\label{thm:DefLRFH}
Consider a cotangent bundle $(T^*Q,\omega)$ with its standard symplectic form and a Hamiltonian $H:T^*Q\to \R$ with regular level set $H^{-1}(0)$.
Fix a pair $q_0,q_1\in Q$ such that both sets $H^{-1}(0)\cap T^*_{q_0}Q$ and $H^{-1}(0)\cap T^*_{q_1}Q$ are compact and nonempty. Fix a compatible almost complex structure $\mathbb{J}\in \mathcal{J}(T^*Q,\omega)$.
Assume that there exists
a compact subset $K$, an open subset $\mathcal{V}$ satisfying $K\subseteq \mathcal{V}\subseteq T^*Q$, and an open neighbourhood $\mathcal{H}\subseteq \{h\in C^\infty(T^*Q)\ |\ dh\in C_0^\infty(K)\}$ of $0$
such that:
\begin{enumerate}
\item for all $h\in \mathcal{H}$ the Hamiltonian $H+h$ satisfies \eqref{T*QcapHnonempty};
\item the Rabinowitz action functional $\A^H_{q_0,q_1}:\mathscr{H}_{q_0,q_1}\times\R \to \R$ has critical set continuously compact in $(K,\mathcal{H})$;
\item for all $h_0,h_1 \in \mathcal{H}$ and $J_0, J_1 \in \mathcal{J}(\mathcal{V}, \mathbb{J})$ every homotopy $\Gamma=\{(H+h_s, J_s)\}_{s\in \R}$ 
from $(H+h_0,J_0)$ to $(H+h_1,J_1)$ satisfies the Novikov finiteness condition \eqref{Novikov}
and for all $\mathbf{a, b} \in \R$ the space of Floer trajectories $\mathcal{M}^\Gamma(\mathbf{a, b})$ is bounded in the $L^\infty$-norm.
\end{enumerate}
Then for every $h\in\mathcal{H}$ the Lagrangian Rabinowitz Floer homology $\LRFH_*(\A^{H+h}_{q_0,q_1})$ is well defined and isomorphic to $\LRFH_*(\A^{H}_{q_0,q_1})$.
\end{thm}

In the next section we will verify that Hamiltonians of the form $H_a-h$ on $(T^*\R^2, \omega_0)$ with $H_a$ as in \eqref{DefHa} and $h\in \mathcal{H}$ as in \eqref{DefH} satisfy the assumptions of Theorem \ref{thm:DefLRFH}:
\begin{itemize}
\item in Lemma \ref{lem:nonEmpty} we prove that for all for all $h\in \mathcal{H}$ the Hamiltonian $H-h$ satisfies \eqref{T*QcapHnonempty};
\item in Lemma \ref{lem:Chord} we prove that there exists a sequence of sets $\{(K_n, \mathcal{H}_n)\}_{n\in \mathbb{N}}$, such that $K_n$ are compact, for all $h\in \mathcal{H}_n$ we have $\supp dh \subseteq K_n$, $\bigcup_{n\in \mathbb{N}} K_n=T^*\R^2$, $\bigcup_{n\in \mathbb{N}} \mathcal{H}_n=\mathcal{H}$ and for every $n\in \mathbb{N}$ the Rabinowitz action functional $\A^{H_a}_{q_0,q_1}:\mathscr{H}_{q_0,q_1}\times\R \to \R$ has critical set continuously compact in $(K_n,\mathcal{H}_n)$;
\item in Theorem \ref{thm:FloerBounds} we prove the boundedness of the space of Floer trajectories $\mathcal{M}^\Gamma(\mathbf{a, b})$ for all $\mathbf{a, b} \in \R$;
\item in Lemma \ref{lem:Novikov} we prove the Novikov finiteness condition.
\end{itemize}
After verifying that Hamiltonians of the form $H_a-h$ satisfy the assumptions of Theorem \ref{thm:DefLRFH}, we can apply it to conclude that Theorem \ref{thm:main} holds true.

\subsection{Levi-Civita regularization}\label{ssec:Levi-Civita}
\label{ssec:reg}
In this subsection we will discuss a method presented in \cite{Levi1920} by Levi-Civita, which shows how to deal with the singularity of the Kepler problem and the corresponding non-compactness of the energy level sets. Recall that the Hamiltonian $K$ of the rotating Kepler problem introduced at \eqref{Kepler} is not defined on the whole $T^*\R^2$, but has a singularity at $0$. The singularity models the possibility of the collision of the spacecraft with the planet. In order to remove the singularity Levi-Civita proposed the following procedure, known as \emph{Levi-Civita regularization}:

First identify $\R^2 \simeq \mathbb{C}$. In complex coordinates the rotating Kepler Hamiltonian on $T^*(\mathbb{C}\setminus \{0\})$ becomes:
$$
K(z,w)=\frac{1}{2}|w|^2+\langle z, iw \rangle-\frac{1}{|z|},
$$
where $\langle\ \cdot\ ,\ \cdot\ \rangle$ is the standard Euclidean scalar product $\langle z , w \rangle=\Re (z \overline{w})$. Consider now the square map $\beta: \mathbb{C} \to \mathbb{C}$ and $\beta(z) = z^2$. Note that $\beta$ restricted to $\mathbb{C}\setminus \{0\}$ is a double-cover of $\mathbb{C}\setminus \{0\}$, which is a local diffeomorphism. The contragradient of the squaring map $\mathfrak{B}: T^*\mathbb{C} \to T^*\mathbb{C}$ is given by
$$
\mathfrak{B}(z,w)=\left(\beta(z), \frac{w}{\overline{\beta'(z)}} \right)=\left(z^2, \frac{w}{2\overline{z}} \right).
$$
For every $z\in \mathbb{C}\setminus\{0\}$ there exists a neighbourhood $\mathcal{U}\subseteq \mathbb{C}\setminus\{0\}$ small enough such that for every connected component of $\beta(\mathcal{U})= \mathcal{V}_+\sqcup V_-$, the sets $T^*\mathcal{V}_{\pm}$ is symplectomorphic with $T^*\mathcal{U}$ via $\mathfrak{B}$. In other words, $\mathfrak{B}$ restricted to $T^*(\mathbb{C}\setminus \{0\})$ is a double cover of $T^*(\mathbb{C}\setminus \{0\})$, which is a local symplectomorphism. 

Let us consider the Hamiltonian
$$
K \circ \mathfrak{B}(z,w)= \frac{|w|^2}{8|z|^2}+ \frac{1}{2}\langle z, iw \rangle-\frac{1}{|z|^2}.
$$
For every $c>0$ the hypersurface $(K \circ \mathfrak{B})^{-1}(c)$ is a double cover of the level set $K^{-1}(c)$. Using $\mathfrak{B}$ we can pull back the Hamiltonian vector field of $K$ to obtain the Hamiltonian vector field of $K \circ \mathfrak{B}$ i.e.
$$
X_{K \circ \mathfrak{B}}=D\mathfrak{B}(X_K).
$$
Moreover, since $\mathfrak{B}$ is a contragradient, for every $z\in \mathbb{C}$ we have $\mathfrak{B}(T^*\mathbb{C})=T^*_{\beta(z)}\mathbb{C}$. Consequently, if we fix $z_0,z_1\in \mathbb{C}\setminus\{0\}$ and 
$(v, \eta)\in \mathscr{H}_{z_0,z_1}\times \R$ is a solution of \eqref{Def(v,eta)} for the Hamiltonian $K\circ \mathfrak{B}$, then $( \mathfrak{B}\circ v, \eta)\in \mathscr{H}_{z_0^2,z_1^2}\times \R$ is the solution of \eqref{Def(v,eta)} for the Hamiltonian $K$. And vice versa: any solution $(v, \eta)\in \mathscr{H}_{z_0^2,z_1^2}\times \R$ of \eqref{Def(v,eta)} for the Hamiltonian $K$ we can lift to two solutions $(v_+,\eta)\in \mathscr{H}_{z_0,z_1}\times \R$ and $(v_-,\eta)\in \mathscr{H}_{-z_0,-z_1}\times \R$ of \eqref{Def(v,eta)} for the Hamiltonian $K\circ\mathfrak{B}$.

Observe that for every $c>0$ the corresponding Hamiltonian $K_c:T^*\mathbb{C}\to\mathbb{R}$ as in \eqref{regHam} satisfies
$$
K_c = 4 |z|^2 \left( K \circ \mathfrak{B}-c\right),
$$
on $T^*(\mathbb{C}\setminus \{0\})$. In other words, $K_c$ is a smooth extension of $4 |z|^2 \left( K \circ \mathfrak{B}-c\right)$ to the whole $T^*\mathbb{C}$. Consequently, we have $(K\circ \mathfrak{B})^{-1}(c)=K_c^{-1}(0)\setminus T^*_0\mathbb{C}$ and
$$
X_{K_c} = 4|z|^2 X_{K \circ \mathfrak{B}}=4|z|^2 D\mathfrak{B}(X_K)\qquad\textrm{on}\quad (K\circ \mathfrak{B})^{-1}(c).
$$
We can view $K_c^{-1}(0)$ as a compactification of the hypersurface $(K\circ \mathfrak{B})^{-1}(c)$. Since $X_{K_c}$ is just a rescaling of $X_{K \circ \mathfrak{B}}$ we have a $1$-to-$1$ correspondence between the solutions of \eqref{Def(v,eta)} on $(K\circ \mathfrak{B})^{-1}(c)$ and the solutions $(v,\eta)$ of \eqref{Def(v,eta)} on $K_c^{-1}(0)$ satisfying $v([0,1])\cap T_0^*\mathbb{C} = \emptyset$. The solutions $(v,\eta)$ of \eqref{Def(v,eta)} on $K_c^{-1}(0)$, such that $v([0,1])\cap T_0^*\mathbb{C} \neq \emptyset$ are called \emph{collision trajectory}.

\begin{rem}
Fix a regular value $c\in \R$ of $K$ and $z_0,z_1\in \mathbb{C}\setminus\{0\}$. Then every solution $(v, \eta)\in \mathscr{H}_{z_0^2,z_1^2}\times \R$ of \eqref{Def(v,eta)} on $K^{-1}(c)$ can be lifted to two solutions $(v_1,\eta_1)\in \mathscr{H}_{z_0,z_1}\times \R$ and $(v_2,\eta_2)\in \mathscr{H}_{-z_0,-z_1}\times \R$ of \eqref{Def(v,eta)} on $K_c^{-1}(0)$. And vice versa: if $(v,\eta)\in \mathscr{H}_{z_0,z_1}\times \R$ is a solution of \eqref{Def(v,eta)} on $K_c^{-1}(0)$ satisfying $v([0,1])\cap T_0^*\mathbb{C}=\emptyset$, then there exists a reparemetrization $\tilde{v}\in \mathscr{H}_{z_0^2,z_1^2}$ of $\beta \circ v $ and $\tilde{\eta}\in \R$, such that $(\tilde{v},\tilde{\eta})$ is a solution of \eqref{Def(v,eta)} on $K^{-1}(c)$.
\end{rem}
We call a finite sequence of maps $\{v_i\}_{i=0}^k$ with a sequence of times $\{t_i\}_{i=1}^k$, $t_i\in (0,1)$ and $\eta\in \R$ a \emph{collision trajectory} between $q_0$ and $q_1$ on $K^{-1}(c)$, if it satisfies the following conditions:
\begin{eqnarray*}
v_0\in W^{1,2}([0,t_1), T^*(\mathbb{C}\setminus\{0\})), & v_0([0,t_1))\subseteq K^{-1}(c), & v_0(0)\in T^*_{q_0}\mathbb{C},\\
v_k \in W^{1,2}((t_k, 1], T^*(\mathbb{C}\setminus\{0\})), & v_k((t_k, 1])\subseteq K^{-1}(c),  & v_k(1)\in T^*_{q_1}\mathbb{C}, \\
v_i\in W^{1,2}((t_i, t_{i+1}), T^*(\mathbb{C}\setminus\{0\})), & v_k((t_i, t_{i+1}))\subseteq K^{-1}(c), & \partial_t v_i = \eta X_K(v_i).
\end{eqnarray*}
Moreover, for $v_i=(z_i,w_i)$ we require
$$
\lim_{t\to t_i}z_{i-1}(t) = \lim_{t\to t_i}z_i(t) = 0,\quad \textrm{and} \quad
\lim_{t\to t_i}\frac{w_{i-1}(t)}{|w_{i-1}(t)|} = -\lim_{t\to t_i}\frac{w_i(t)}{|w_i(t)|}.
$$

\begin{rem}
If $(v,\eta)\in \mathscr{H}_{z_0,z_1}\times \R$ is a solution of \eqref{Def(v,eta)} on $K_c^{-1}(0)$ satisfying $v([0,1])\cap T_0^*\mathbb{C}\neq\emptyset$, then there exists a unique collision trajectory between $z_0^2$ and $z_1^2$ on $K^{-1}(c)$.
\end{rem}

To conclude: to understand solutions of \eqref{Def(v,eta)} on $K^{-1}(c)$ we can analyse solutions of \eqref{Def(v,eta)} on $K_c^{-1}(0)$, instead.

\subsection{Properties of Hamiltonians}
In this subsection we briefly discuss properties of Hamiltonians introduced in \eqref{DefHa} and \eqref{DefH}, which will be crucial later on in establishing uniform bounds on he set of Floer trajectories.

First observe that for a Hamiltonian $H_a$ as in \eqref{DefHa} we have
$$
dH_a(p\partial_p)-H_a=\frac{1}{2}|p|^2.
$$
Consequently, for every Hamiltonian $h\in \mathcal{H}$ satisfying \eqref{DefH} we have
\begin{equation}\label{d(H_a-h)}
d(H_a-h)(p\partial_p)-(H_a-h)=\frac{1}{2}|p|^2+h-dh(p\partial_p)>0.
\end{equation}
In particular, the hypersurface $(H_a-h)^{-1}(0)$ is of contact type.

Now, for $h\in \mathcal{H}$ satisfying \eqref{DefH} we introduce two corresponding constants:
\begin{align}
\alpha & := \inf\left( h-dh(p\partial_p)+\frac{1}{2}|p|^2\right) >0\label{DefAlpha}\\
\beta & := \max\left\lbrace  \sup h, \sup (dh(p\partial_p)-h), -\inf qdq(X_h), \sup pdp(X_h) \right\rbrace\geq 0 \label{DefBeta}
\end{align}

\section{Bounds for Floer trajectories}
\label{sec:Bounds}

The aim of this section is to show that the Hamiltonian $H_a$ defined in \eqref{DefHa} together with the set of compactly supported perturbations defined in \eqref{DefH} satisfy the assumptions of Theorem \ref{thm:DefLRFH}, in order to apply it and prove that $\LRFH(\A^{H_a-h})$ is well defined for every $h\in \mathcal{H}$.

We will start by proving in Lemma \ref{lem:nonEmpty} that assumption \eqref{T*QcapHnonempty} holds true in our setting.
If it weren't, our search for Reeb chords would have been futile.
Next, in Lemma \ref{lem:Chord} we will prove that the second assumption of Theorem \ref{thm:DefLRFH} holds true in our setting, i.e. the critical set of Rabinowitz action functional is continuously compact.

The most challenging part, which will occupy most of this section, is to prove that the Floer trajectories are uniformly bounded in the $L^\infty$-norm. This is essential in defining the Lagrangian Rabinowitz homology and showing that it is invariant of the perturbation $h\in \mathcal{H}$. We will prove it in Theorem \ref{thm:FloerBounds}.

\begin{lem}\label{lem:nonEmpty}
Let $H_a$ be the Hamiltonian defined in \eqref{DefHa} and let $\mathcal{H}$ be the set of perturbations defined in \eqref{DefH}. Then for every $h\in \mathcal{H}$  the set $(H_a-h)^{-1}(0)\cap T_q\R^2$ is compact and non-empty for every $q\in \R^2$.
\end{lem}

\begin{proof}
Fix $q\in \R^2$. Since $h$ is constant outside a compact set,
we have
$$
\lim_{|p|\to\infty} \left(H_a(q,p)-h(q,p)\right)=+\infty.
$$
Moreover, we have $H_a(q,0)=0$ and $h(q,0)>0$, since $h\in \mathcal{H}$. Thus
$$
H_0(q,0)-h(q,0)=-h(q,0)<0.
$$
Consequently, by the intermediate value theorem for every $p\in S_q^*\R^2$ there exists $\lambda_p\in \R_+$, such that $H_a(q,\lambda_p p)-h(q,\lambda_p p)=0$. Moreover, by \eqref{d(H_a-h)} the $\lambda_p\in \R_+$ is unique. In other words, $(H_a-h)^{-1}(0)\cap T_q\R^2$ is a circle $S^1$.
\end{proof}

\begin{lem}\label{lem:Chord}
Let $H_a$ be the Hamiltonian defined in \eqref{DefHa} and let $\mathcal{H}$ be the set of perturbations defined in \eqref{DefH}. Then the critical set of the associated Rabinowitz action functional $\mathcal{A}^{H_a-h}$ is continuously compact.
\end{lem}

\begin{proof}
Fix $q_0,q_1\in \R^2$, $R\geq \max\{|q_0|,|q_1|\}$ and $\mathfrak{h}>0$. We will show that the Rabinowitz action functional $\A^{H_a-h}$ has the critical set continuously compact for in the sets
\begin{align*}
K_{R, \mathfrak{h}} & :=
\left\lbrace\begin{array}{c|c}
& r\leq R, \quad p_r \leq \sqrt{R^2+2\mathfrak{h}},\\
\hspace*{-0.2cm}\smash{\raisebox{.5\normalbaselineskip}{ $(r, p_r,\theta,p_\theta)\in T^*\R_+\times T^*S^1$}}& p_\theta \leq 2R\sqrt{R^2+2\mathfrak{h}}.
\end{array}\right\rbrace,\\
\mathcal{H}_{\mathfrak{h}} & := \left\lbrace h \in \mathcal{H}\ \Big|\ \supp dh \subseteq K_{R, \mathfrak{h}}, \quad \|h\|_{L^\infty}\leq \mathfrak{h}\right\rbrace
\end{align*}

Fix $h \in \mathcal{H}_{\mathfrak{h}}$ and let $(v, \eta) \in \Crit \A^{H_0-h}_{q_0,q_1}$. Suppose that $r\circ v(t_0)=\max r \circ v$ for some $t_0\in (0,1)$. Then 
\begin{align*}
\frac{d}{dt}r\circ v (t_{0}) & = \eta\{H_a-h, r\}\circ  v (t_{0})= 0, \\
\frac{d^{2}}{dt^{2}}r\circ v (t_{0}) & =\eta^2\{H_a-h, \{H_a-h, r\}\} \circ v (t_{0})\leq 0,
\end{align*}
where $\{\ ,\ \}$ stands for the Poisson bracket. In other words,
\begin{equation}\label{cpctCord}
v(t_0)\in (H_a-h)^{-1}(0)\cap \{\{H_a-h\},r\}=0\}\cap \{ \{H_a-h,\{H_a-h,r\}\}\leq 0\}.
\end{equation}

We will first show that $\max r \circ v \leq R$.
Arguing by contradiction, suppose that $r\circ v(t_0)=\max r \circ v > R$ for some $t_0\in [0,1]$.  
Since $r\circ v(i)=|q_i|\leq R$ for $i=0,1$, we have $t_0\in (0,1)$ and $v(t_0)$ satisfies \eqref{cpctCord}.
By assumption $v(t_0) \notin K_{R, \mathfrak{h}}$ and $\supp h\subseteq K_{R, \mathfrak{h}}$, hence $h$ does not contribute to the equations above and we have
\begin{align*}
\{H_a-h, r\}\circ v (t_{0}) & = \{H_a, r\}\circ  v (t_{0}) =p_r\circ v(t_0), \\
\{H_a-h,\{H_a-h, r\}\}\circ v (t_{0}) & =\{H_a, \{H_a, r\}\} \circ v (t_{0})= \frac{p_\theta^2}{r^3}\circ v(t_0).
\end{align*}
Moreover, $h|_{T^*\R^2\setminus K_{R,\mathfrak{h}}}\equiv c$ for some positive constant $c\in (0, \mathfrak{h})$. Thus \eqref{cpctCord} becomes
$$
v(t_0)\in H_a^{-1}(c)\cap \{p_r=0\}\cap \{ p_\theta=0\}\subseteq H_a^{-1}(c)\cap H_a^{-1}(0) = \emptyset,
$$
which brings us a contradiction we were looking for.

Now we will show that $(H-h)^{-1}(0)\cap T^*B(R) \subseteq K_{R,\mathfrak{h}}$, where $B(R)\subseteq \R^2$ is a ball of radius $R$ centred at the origin. Let us calculate
\begin{align*}
\frac{p_r^2}{2}+\frac{p_\theta^2}{2r^2} +p_\theta & =h(r,\theta, p_r, p_\theta),\\
p_r^2+\left(\frac{p_\theta}{r}+r\right)^2 & =r^2+2h  \leq R^2+2\mathfrak{h},\\
|p_r \circ v|  & \leq \sqrt{R^2+2\mathfrak{h}},\\
|p_\theta \circ v| & \leq R\left(\sqrt{R^2+2\mathfrak{h}}+R\right)\\
&\leq 2R\sqrt{R^2+2\mathfrak{h}}.
\end{align*}
Thus for all $t \in [0,1]$ we have $v(t)\in (H-h)^{-1}(0)\cap T^*B(R) \subseteq K_{R,\mathfrak{h}}$.
\end{proof}
Now we will present the main theorem of this section:
\begin{thm}\label{thm:FloerBounds}
Let $H_a$ be the Hamiltonian defined in \eqref{DefHa} and let $h_0,h_1 \in \mathcal{H}$. Fix two points $q_0,q_1 \in \R^2$. Let $K$ be a compact subset of $T^*\R^2$, such that $q_0,q_1\in K$ and $\supp dh_i \subseteq K$.
Let $\mathcal{V}\subseteq T^*\R^2$ be an open, but precompact subset, such that $K\subseteq \mathcal{V}$. Let $\Gamma:=\{(h_s, J_s)\}_{s\in \R}$ be a smooth homotopy of the Hamiltonians $h_s\in \mathcal{H}$ and $2$-parameter families of almost complex structures $J_s \in C^\infty([0,1]\times \R,\mathcal{J}(\mathcal{V},\mathbb{J}))$, such that
\begin{enumerate}[label=\alph*)]
\item the homotopy is constant in $s$ outside $[0,1]$,
\item we have $\supp dh_s \subseteq K$ for all $s\in \R$,
\item the Hamiltonians  satisfy
\begin{equation}\label{CondGamma}
\|\partial_sh_s\|_{L^\infty}\left( \frac{1}{\alpha}+\sqrt{2(\alpha+\beta)}\|J\|_{L^{\infty}}^2\right) \leq \frac{1}{3}.
\end{equation}

\end{enumerate}

Then for every pair $(\mathbf{a},\mathbf{b})\in \R^2$ the corresponding space $\mathcal{M}^\Gamma(\mathbf{a},\mathbf{b})$ of Floer trajectories defined in \eqref{DefM(a,b)} is bounded in the $L^\infty$-norm.
\end{thm}

\subsection{The set of infinitesimal action derivation}
For a Hamiltonian $H:T^*Q\to \R$ and a positive constant $\varepsilon>0$ we define the following set:
\begin{equation}
\mathcal{B}(H,\varepsilon):= \left\lbrace\begin{array}{c|c}
(v, \eta) \in\mathscr{H}_{q_0,q_1}\times\R & \|\nabla\A^H(v,\eta)\|_{L^2\times\R}<\varepsilon.
\end{array}\right\rbrace
\end{equation}
We will call this set the set of infinitesimally small action derivation.

\begin{prop}\label{prop:Infinitesimal}
Fix $a>0$ and let $H_a$ be a Hamiltonian as in \eqref{DefHa}. Fix $h\in \mathcal{H}$ and let $\alpha>0$ and $\beta\geq 0$ be constants as in \eqref{DefAlpha} and \eqref{DefBeta}, respectively.
Then for every $\varepsilon>0$, whenever $(v,\eta)\in \mathcal{B}(H_a-h,\varepsilon)$ then it satisfies the following inequalities:
\begin{align*}
|\eta| & \leq \frac{1}{\alpha}\left(|\A^{H_a-h}(v,\eta)|+\varepsilon \sqrt{2(\alpha+\beta)}\right),\\
\|q\|_{L^2} & \leq \min\{|q_0|, |q_1|\}+\frac{2}{\alpha}\left(\varepsilon(3\alpha+2\beta) +\left|\A^{H_a-h}(q,p,\eta)\right|\sqrt{2(\alpha+\beta)}\right)\\
& +\sqrt{\frac{2\beta}{\alpha}\left(\left|\A^{H_a-h}(q,p,\eta)\right|+\varepsilon \sqrt{2(\alpha+\beta)}\right)},\\
\|p\|_{L^2} & \leq 2|a| \min\{|q_0|, |q_1|\}+\frac{4|a|}{\alpha}\left(\varepsilon(3\alpha+2\beta) +\left|\A^{H_a-h}(q,p,\eta)\right|\sqrt{2(\alpha+\beta)}\right)\\
& +2|a|\sqrt{\frac{2\beta}{\alpha}\left(\left|\A^{H_a-h}(q,p,\eta)\right|+\varepsilon \sqrt{2(\alpha+\beta)}\right)}+\sqrt{2(\varepsilon+\beta)}.
\end{align*}
\end{prop}
We will prove the proposition in several lemmas:

\begin{lem}\label{lem:BoundEta}
Consider the setting as in Proposition \ref{prop:Infinitesimal}. Then for every $\varepsilon>0$, whenever $(v,\eta)\in \mathcal{B}(H_a-h,\varepsilon)$ then
$$
|\eta|\leq \frac{1}{\alpha}\left(\left|\A^{H_a-h}(v,\eta)\right|+\varepsilon \sqrt{2(\alpha+\beta)}\right).
$$
\end{lem}
\begin{proof}
Note that by \eqref{d(H_a-h)} and \eqref{DefBeta} we have 
$$
d(H_a-h)(p\partial_p)-(H_a-h)\geq \frac{1}{2}|p|^2-\beta, \qquad \forall\ (q,p)\in T^*\R^2.
$$

Take $(q,p, \eta)\in \mathcal{B}(H_a-h,\varepsilon)$ and denote $A:=\left|\A^{H_a-h}(q,p,\eta)\right|$. Combining equality \eqref{dA^H(Y)} with the inequality above we obtain
$$
A+\varepsilon \|p\|_{L^2} \geq |\eta|\left( \frac{1}{2}\|p\|^2_{L^2} -\beta \right).
$$
On the other hand, combining equality \eqref{dA^H(Y)} with \eqref{d(H_a-h)} and \eqref{DefAlpha} we obtain 
$$
A+\varepsilon \|p\|_{L^2} \geq \alpha|\eta|.
$$
Now, if we denote
\begin{gather}
f(x) :=\min\left\lbrace \frac{1}{\alpha}\left( A+\varepsilon x\right), \frac{A+\varepsilon x}{\frac{1}{2}x^2-\beta}\right\rbrace,\label{Deff} \\
\textrm{then} \qquad |\eta| \leq \max_{x\geq 0} f(x).\nonumber
\end{gather}
Let us calculate
\begin{align*}
f(x) & = \begin{cases}
 \frac{1}{\alpha}\left( A+\varepsilon x\right) & \textrm{for} \quad x \in \left[0, \sqrt{2(\alpha+\beta)}\right]\\
 \frac{A+\varepsilon x}{\frac{1}{2}x^2-\beta} & \textrm{for} \quad x \geq \sqrt{2(\alpha+\beta)}
\end{cases}\\
f'(x) & = \begin{cases}
 \frac{\varepsilon}{\alpha} & \textrm{for} \quad x \in \left[0, \sqrt{2(\alpha+\beta)}\right)\\
 -\frac{\frac{\varepsilon}{2}x^2+A+\varepsilon \beta}{\left(\frac{1}{2}x^2-\beta\right)^2} & \textrm{for} \quad x > \sqrt{2(\alpha+\beta)}
\end{cases}\\
\textrm{Since}\qquad f'(x) & >0\qquad \textrm{for}\qquad x \in \left[0, \sqrt{2(\alpha+\beta)}\right),\\
 \textrm{and}\qquad f'(x) & <0\qquad \textrm{for}\qquad x > \sqrt{2(\alpha+\beta)},\\
  \textrm{therefore} \qquad \max_{x\geq 0} f(x) & = f\left(\sqrt{2(\alpha+\beta)}\right)=\frac{1}{\alpha}\left(A+\varepsilon \sqrt{2(\alpha+\beta)}\right),
\end{align*}
which concludes the proof.
\end{proof}

\begin{lem}\label{lem:BoundEtaP}
Consider the setting as in Proposition \ref{prop:Infinitesimal}. 
If $(q,p,\eta)\in\linebreak \mathcal{B}(H_a-h, \varepsilon)$ then
$$
|\eta|\|p\|_{L^2} \leq \frac{1}{\alpha}\left(2\varepsilon(\alpha+\beta) +\left|\A^{H_a-h}(q,p,\eta)\right|\sqrt{2(\alpha+\beta)}\right).
$$
\end{lem}
\begin{proof}
Let $f:\R_+\to \R_+$ be the continuous function defined in \eqref{Deff}. Then for all $(q,p,\eta) \in \mathcal{B}(H_a-h, \varepsilon)$ we have 
$$
|\eta|\|p\|_{L^2} \leq \max_{x\geq 0} xf(x).
$$
We can calculate
\begin{align*}
\frac{d}{dx}\left(\frac{Ax+\varepsilon x^2}{\frac{1}{2}x^2-\beta} \right) & = \frac{(A+2\varepsilon)(\frac{1}{2}x^2-\beta)-x^2(A+\varepsilon x)}{\left( \frac{1}{2}x^2-\beta\right)^2}\\
& = -\frac{\frac{1}{2}\mathfrak{a}x^2+2 \beta \varepsilon x+\beta A}{\left( \frac{1}{2}x^2-\beta\right)^2}<0,\\
\max_{x\geq 0} xf(x) & = \frac{1}{\alpha}\left(2\varepsilon(\alpha+\beta) +A\sqrt{2(\alpha+\beta)}\right).
\end{align*}
\end{proof}

\begin{lem}\label{lem:BoundQ}
Consider the setting as in Proposition \ref{prop:Infinitesimal}. 
If $(q,p,\eta)\in\linebreak \mathcal{B}(H_a-h, \varepsilon)$, then
\begin{align*}
\|q\|_{L^2}, \|q\|_{L^\infty} & \leq \min\{|q_0|, |q_1|\}+\frac{2}{\alpha}\left(\varepsilon(3\alpha+2\beta) +\left|\A^{H_a-h}(q,p,\eta)\right|\sqrt{2(\alpha+\beta)}\right)\\
& +\sqrt{\frac{2\beta}{\alpha}\left(\left|\A^{H_a-h}(q,p,\eta)\right|+\varepsilon \sqrt{2(\alpha+\beta)}\right)}.
\end{align*}
\end{lem}
\begin{proof}
Using \eqref{DefHa} and \eqref{DefBeta} we can estimate
\begin{align*}
qdq(X_{H_a-h}) & =  q_1(p_1+aq_2)+q_2(p_2-aq_1)-qdq(X_h)\\
& =q_1p_1+q_2p_2-qdq(X_h)\leq |q||p|+\beta,
\end{align*}
for every $(q,p)\in T^*\R^2$.

Now, let $(q,p,\eta)\in \mathcal{B}(H_a-h, \varepsilon)$. Using first the inequality above and then the bounds Lemma \ref{lem:BoundEta} and Lemma \ref{lem:BoundEtaP}, we can estimate that for all $t\in [0,1]$ we have \begin{align}
|q(t)|^2-|q_0|^2 & = 2 \int_0^t qdq(\partial_tv)=2\int_0^t qdq(\partial_tv-\eta X_H(v))+ 2\eta\int_0^t qdq(X_H)\nonumber\\
& \leq 2 \|q\|_{L^2}\|\nabla \A^H(v,\eta)\|_{L^2\times \R}+2|\eta|\left(\|q\|_{L^2}\|p\|_{L^2}+\beta\right)\nonumber\\
& \leq 2 \varepsilon \|q\|_{L^2}+\frac{2}{\alpha}\|q\|_{L^2}\left(2\varepsilon(\alpha+\beta) +A\sqrt{2(\alpha+\beta)}\right)\nonumber\\
& + \frac{2\beta}{\alpha}\left(A+\varepsilon \sqrt{2(\alpha+\beta)}\right)\nonumber\\
&\leq \frac{2}{\alpha} \|q\|_{L^2}\left(\varepsilon(3\alpha+2\beta) +A\sqrt{2(\alpha+\beta)}\right)+ \frac{2\beta}{\alpha}\left(A+\varepsilon \sqrt{2(\alpha+\beta)}\right)\label{eq1q}.
\end{align}
Integrating both sides we obtain
$$
\|q\|_{L^2}^2 \leq |q_0|^2+\frac{2}{\alpha} \|q\|_{L^2}\left(\varepsilon(3\alpha+2\beta) +A\sqrt{2(\alpha+\beta)}\right)+ \frac{2\beta}{\alpha}\left(A+\varepsilon \sqrt{2(\alpha+\beta)}\right)
$$
By solving this quadratic inequality, we obtain the following bound
\begin{align*}
\|q\|_{L^2} & \leq \frac{1}{\alpha}\left(\varepsilon(3\alpha+2\beta) +A\sqrt{2(\alpha+\beta)}\right)\\
&+\sqrt{\frac{1}{\alpha^2}\left(\varepsilon(3\alpha+2\beta) +A\sqrt{2(\alpha+\beta)}\right)^2+|q_0|^2+ \frac{2\beta}{\alpha}\left(A+\varepsilon \sqrt{2(\alpha+\beta)}\right)}\\
& \leq |q_0|+\frac{2}{\alpha}\left(\varepsilon(3\alpha+2\beta) +A\sqrt{2(\alpha+\beta)}\right)+\sqrt{\frac{2\beta}{\alpha}\left(A+\varepsilon \sqrt{2(\alpha+\beta)}\right)}.
\end{align*}
By repeating this procedure with the equation $|q(t)|^2=|q_1|^2-2\int_t^1 q\partial_tq$, we obtain the bound on $\|q\|_{L^2}$ we were looking for.

To obtain the bound on $\|q\|_{L^\infty}$ we use \eqref{eq1q} again:
\begin{align*}
\|q\|_{L^\infty} & \leq \sqrt{|q_0|^2+\frac{2}{\alpha} \|q\|_{L^2}\left(\varepsilon(3\alpha+2\beta) +A\sqrt{2(\alpha+\beta)}\right)+ \frac{2\beta}{\alpha}\left(A+\varepsilon \sqrt{2(\alpha+\beta)}\right)}\\
& \leq |q_0|+\frac{2}{\alpha}\left(\varepsilon(3\alpha+2\beta) +A\sqrt{2(\alpha+\beta)}\right)+\sqrt{\frac{2\beta}{\alpha}\left(A+\varepsilon \sqrt{2(\alpha+\beta)}\right)}.
\end{align*}
Analogously as before, we repeat this procedure with the equation\linebreak $|q(t)|^2=|q_1|^2-2\int_t^1 q\partial_tq$ to obtain the bound on $\|q\|_{L^\infty}$ we were looking for.
\end{proof}
\begin{lem}
Consider the setting as in Proposition \ref{prop:Infinitesimal}. 
Whenever $(q,p,\eta)\in\linebreak \mathcal{B}(H_a-h, \varepsilon)$, then
\begin{align*}
\|p\|_{L^2} &\leq 2\left(|a|\mathfrak{q}+\sqrt{\varepsilon+\beta}\right),\\
\|p\|_{L^\infty} & \leq \sqrt{4\left( |a|\mathfrak{q}+\sqrt{\varepsilon+\beta}\right)(\varepsilon+|a|\mathfrak{q})+ 2(\varepsilon+\beta)+\frac{2\beta}{\alpha}\left(A+\varepsilon\sqrt{2(\alpha+\beta)}\right)}\\
\textrm{where}\qquad A & := \left|\A^{H_a-h}(q,p,\eta)\right|,\\
\textrm{and}\qquad \mathfrak{q} & := \min\{|q_0|, |q_1|\}+\frac{2}{\alpha}\left(\varepsilon(3\alpha+2\beta) +A\sqrt{2(\alpha+\beta)}\right) +\sqrt{\frac{2\beta}{\alpha}\left(A+\varepsilon\sqrt{2(\alpha_\beta)}\right)}.
\end{align*}

\end{lem}
\begin{proof}
Using \eqref{DefHa} and \eqref{DefBeta} we can calculate
\begin{align*}
\frac{1}{2}\|p\|_{L^2}^2 & \leq \int_0^1 H_a(v(t))dt+|a|\|q\|_{L^2}\|p\|_{L^2}\\
& \leq\left| \int_0^1 (H_a-h)(v(t))dt \right| +|a|\|q\|_{L^2}\|p\|_{L^2}+ \sup h\\
& \leq \varepsilon +|a|\|q\|_{L^2}\|p\|_{L^2}+\beta.
\end{align*}
If we now denote
$$
\mathfrak{q}:=\min\{|q_0|,|q_1|\}+\frac{2}{\alpha}\left(\varepsilon(3\alpha+2\beta) +A\sqrt{2(\alpha+\beta)}\right)+\sqrt{\frac{2\beta}{\alpha}\left(A+\varepsilon \sqrt{2(\alpha+\beta)}\right)},
$$
then by Lemma \ref{lem:BoundQ} we have
$$
\frac{1}{2}\|p\|_{L^2}^2 \leq |a|\mathfrak{q}\|p\|_{L^2}+\varepsilon+\beta.
$$
Solving this inequality we obtain
\begin{equation}\label{pL2Bound}
\|p\|_{L^2} \leq |a|\mathfrak{q}+\sqrt{a^2\mathfrak{q}^2+4(\varepsilon+\beta)}\leq 2|a|\mathfrak{q}+2\sqrt{\varepsilon+\beta},
\end{equation}
which proves the first bound.

To prove the second bound observe that there exists $t_0 \in [0,1]$ such that $|p(t_0)|\leq \|p\|_{L^2}$. Additionally, we $pdp(X_{H_a})=0$,  and thus for any $t\in [0,1]$ we can estimate
\begin{align*}
|p(t)|^2 & = |p(t_0)|^2+ 2\int_{t_0}^t pdp(\partial_tv)\\
& = |p(t_0)|^2+ 2\int_{t_0}^tpdp(\partial_tv-\eta X_{H_a-h}(v))+2\eta \int_{t_0}^t pdp(X_h)\\
& \leq \|p\|_{L^2}^2+2\varepsilon\|p\|_{L^2}+2\beta|\eta|\\
& \leq 2\|p\|_{L^2}(\varepsilon+|a|\mathfrak{q})+ 2(\varepsilon+\beta)+2\beta|\eta|\\
& \leq 4\left( |a|\mathfrak{q}+\sqrt{\varepsilon+\beta}\right)(\varepsilon+|a|\mathfrak{q})+ 2(\varepsilon+\beta)+\frac{2\beta}{\alpha}\left(A+\varepsilon\sqrt{2(\alpha+\beta)}\right)
\end{align*}
where the last inequality uses the bounds obtained in \eqref{pL2Bound} and in Lemma \ref{lem:BoundEta}.
\end{proof}

\begin{rem}\label{rem:HomConst}
Note that if we consider a smooth homotopy of Hamiltonians $\{h_s\}_{s\in [0,1]}$, such that $\supp dh_s \subseteq K$, then we have
\begin{equation}
\begin{aligned}
\alpha & := \inf_{s\in [0,1]}\inf_{T^*\R^2}\left( h-dh(p\partial_p)+\frac{1}{2}|p|^2\right) >0,\\
\beta & := \sup_{s\in[0,1]}\max\left\lbrace  \sup h, \sup (dh(p\partial_p)-h), -\inf qdq(X_h), \sup pdp(X_h) \right\rbrace\geq 0.
\end{aligned}\nonumber
\end{equation}
In particular, if we take $(v,\eta)\in \bigcup_{s\in[0,1]}\mathcal{B}(H_a-h_s,\varepsilon)$, then it satisfies the same bounds as in Proposition \ref{prop:Infinitesimal} with constants $\alpha>0$ and $\beta\geq 0$ as above.

\end{rem}

\begin{rem}\label{rem:BGammaBound}
Note that if we consider a smooth homotopy of Hamiltonians $\{h_s\}_{s\in [0,1]}$, such that $\supp dh_s \subseteq K$, then for fixed $\varepsilon, \mathfrak{a}>0$ the set
\begin{equation}
\mathcal{B}^\Gamma(\mathfrak{a},\varepsilon) := \left\lbrace\begin{array}{c|c}
(v,\eta)\in \mathcal{B}(H_a-h_s, \varepsilon) & |\mathcal{A}^{H_a-h_s}(v,\eta)|\leq \mathfrak{a}\quad s\in [0,1]
\end{array}\right\rbrace,
\end{equation}
is bounded in the $L^2\times\R$-norm and $L^\infty\times\R$-norm.
\end{rem}

\subsection{Bounds for the energy and action}

\begin{lem}\label{lem:ActBound}
Consider the setting as in Theorem \ref{thm:FloerBounds}. Fix $\mathbf{a},\mathbf{b}\in \mathbb{R}$. Then $\|\eta\|_{L^\infty}$, $\|\A^{H_0-h_s}\circ u(s)\|_{L^\infty}$ and $\|\nabla^{J_s}\A^{H_0-h_s}\circ u(s)\|_{L^2([0,1])\times\R}$ are uniformly bounded for all $u \in \mathcal{M}^\Gamma(\mathbf{a},\mathbf{b})$.
\end{lem}
The following proof follows closely the arguments presented in \cite[Prop. 3.3]{Pasquotto2017}. We will not present all of them here, but we encourage a curious reader to see the details in \cite{Pasquotto2017}.
\begin{proof}
Let $\alpha>0$ and $\beta\geq 0$ be the constants associated to the homotopy $\{h_s\}_{s\in[0,1]}$ as in Remark \ref{rem:HomConst}.

Using the fact that $u\in \mathcal{M}^\Gamma(\mathbf{a},\mathbf{b})$ is a Floer trajectory, one can calculate the derivative of the action functional over $s$ and obtain the following inequalities (see  \cite[Prop. 3.3]{Pasquotto2017}):
\begin{align}
\|\mathcal{A}^{H_0-h_{s}}( u)\|_{L^{\infty}} & \leq M + \|\eta\|_{L^{\infty}}\|\partial_{s}h_{s}\|_{L^{\infty}},\label{inqAct}\\
\|\nabla^{J_{s}} \mathcal{A}^{H_0-h_{s}}(u)\|_{L^{2}(\mathbb{R}\times [0,1])}^{2}
& \leq \|J\|_{L^{\infty}} (L + \|\eta\|_{L^{\infty}}\|\partial_{s}h_{s}\|_{L^{\infty}}) \label{inqEnergy}\\
\textrm{where}\qquad M := \max\{|\mathbf{a}|,|\mathbf{b}|\}\qquad &\textrm{and}\qquad L:= \mathbf{b}-\mathbf{a}.\nonumber
\end{align}

In particular, the convergence of the integral
$$
\|\nabla \mathcal{A}^{H_0-h_{s}}(u)\|_{L^{2}(\mathbb{R}\times [0,1])} \leq \|J\|_{L^{\infty}}  \|\nabla^{J_{s}} \mathcal{A}^{H_0-h_{s}}(u)\|_{L^{2}(\mathbb{R}\times [0,1])},
$$
implies that for small enough $s$ we have $\|\nabla \mathcal{A}^{H_0-h_{s}}(u(s))\|_{L^{2}\times\mathbb{R}}<(2(\alpha+\beta))^{-\frac{1}{2}}$.
This ensures that for all $s \in \mathbb{R}$ the following value $\tau_{0}(s)$ is well defined and finite
$$
\tau_{0}(s): = \sup\left\lbrace \tau \leq s\ \Big|\ \|\nabla \mathcal{A}^{H_0-h_{\tau}}(u(\tau))\|_{L^{2}\times\mathbb{R}}<\frac{1}{\sqrt{2(\alpha+\beta)}}\right\rbrace.
$$
For $\tau \in [\tau_{0}(s),s]$ we have $\|\nabla \mathcal{A}^{H_0-h_{\tau}}(u(\tau))\|_{L^{2}\times\mathbb{R}}\geq (2(\alpha+\beta))^{-\frac{1}{2}}$, which allows us to estimate (see  \cite[Prop. 3.3]{Pasquotto2017}):
\begin{align}
|s -\tau_0(s)| & \leq 2(\alpha +\beta)\|J\|_{L^{\infty}}^3 (L + \|\eta\|_{L^{\infty}}\|\partial_{s}h_{s}\|_{L^{\infty}}),\nonumber\\
|\eta(s) -\eta(\tau_0(s))| & \leq \sqrt{2(\alpha+\beta)}\|J\|_{L^{\infty}}^2 (L + \|\eta\|_{L^{\infty}}\|\partial_{s}h_{s}\|_{L^{\infty}}). \label{inqeta}
\end{align}
Using the estimates from Lemma \ref{lem:BoundEta} together with \eqref{inqAct} and \eqref{inqeta} we can estimate
\begin{align}
|\eta(s)|& \leq |\eta(\tau_0(s))|+|\eta(s) -\eta(\tau_0(s))|\nonumber\\
& \leq \frac{1}{\alpha}\left(|\A^{H_0-h_s}(u(s))|+1\right)+|\eta(s) -\eta(\tau_0(s))|\nonumber\\
& \leq \frac{1}{\alpha}\left(M + \|\eta\|_{L^{\infty}}\|\partial_{s}h_{s}\|_{L^{\infty}}+1\right)+\sqrt{2(\alpha+\beta)}\|J\|_{L^{\infty}}^2 (L + \|\eta\|_{L^{\infty}}\|\partial_{s}h_{s}\|_{L^{\infty}})\nonumber\\
& \leq  \|\eta\|_{L^{\infty}}\|\partial_{s}h_{s}\|_{L^{\infty}}\left( \frac{1}{\alpha}+\sqrt{2(\alpha+\beta)}\|J\|_{L^{\infty}}^2\right)+\frac{M+1}{\alpha}+L\sqrt{2(\alpha+\beta)}\|J\|_{L^{\infty}}^2\nonumber\\
\|\eta\|_{L^{\infty}} & \leq \frac{\frac{1}{\alpha}\left(M+1\right)+L \sqrt{2(\alpha+\beta)}\|J\|_{L^{\infty}}^2}{1-\|\partial_{s}h_{s}\|_{L^{\infty}}\left( \frac{1}{\alpha}+\sqrt{2(\alpha+\beta)}\|J\|_{L^{\infty}}^2\right)}\nonumber\\
& \leq \frac{3}{2}\left( \frac{1}{\alpha}\left(M+1\right)+L \sqrt{2(\alpha+\beta)}\|J\|_{L^{\infty}}^2\right)=: \mathfrak{y}.\label{eqEta}
\end{align}
Now using \eqref{CondGamma}, \eqref{inqAct}, \eqref{inqEnergy} and \eqref{eqEta} we obtain the desired uniform bounds:
\begin{align}
\|\mathcal{A}^{H_0-h_{s}}( u)\|_{L^{\infty}} &\leq M + \frac{\alpha\mathfrak{y}}{3}=:\mathfrak{a},\label{DefA}\\ 
\|\nabla^{J_{s}} \mathcal{A}^{H_0-h_{s}}(u)\|_{L^{2}(\mathbb{R}\times [0,1])}^{2} &  \leq \|J\|_{L^{\infty}} \left(L + \frac{\alpha\mathfrak{y}}{3}\right) =:\mathfrak{e}.\label{DefE}
\end{align}
\end{proof}

Having obtained the bounds on the action, we are ready to prove the Novikov finiteness condition:
\begin{lem}\label{lem:Novikov}
Consider the setting as in Theorem \ref{thm:FloerBounds}. Then $\mathcal{M}_\Gamma(\mathbf{a},\mathbf{b})\neq \emptyset$ implies
$$
\mathbf{a} \leq \max\left\lbrace 3 \mathbf{b}, 3
\right\rbrace \qquad \textrm{and} \qquad \mathbf{b} \geq \min\left\lbrace 3\mathbf{a}, -3\right\rbrace.
$$
\end{lem}

\begin{proof}
This proof follows arguments presented in \cite[Cor. 3.8]{CieliebakFrauenfelder2009}. We will prove the first inequality since the second is completely analogous.
\begin{align}
\textrm{First assume}\ \qquad  & 3\leq \mathbf{a} \quad \textrm{and} \quad |\mathbf{b}|\leq \mathbf{a}.\label{aBound}\\
\textrm{Then} \qquad M:= \max\{|\mathbf{a}|,|\mathbf{b}|\} & = \mathbf{a} \quad \textrm{and} \quad L := \mathbf{b}-\mathbf{a}\leq 0.\nonumber
\end{align}
By \eqref{eqEta} and \eqref{aBound} we get:
$$
\|\eta\|_{L^\infty}\leq \frac{3}{2\alpha}(\mathbf{a}+1)\leq \frac{2}{\alpha}\mathbf{a}.
$$
On the other hand, by \eqref{CondGamma} we have $\|\partial_sh_s\|_{L^\infty}\leq \frac{\alpha}{3}$. This, together with \eqref{inqEnergy}, implies
$$
\mathbf{b} \geq \mathbf{a}- \|\eta\|_{L^\infty}\|\partial_sh_s\|_{L^\infty} \geq \mathbf{a} -\frac{2}{3}\mathbf{a}=\frac{1}{3}\mathbf{a}.
$$

Now assume that  $ 3 \leq \mathbf{a} < |\mathbf{b}|$. To finish the proof it suffices to exclude the case $\mathbf{a}< - \mathbf{b}$.
$$
\textrm{Then} \qquad M:=\max\{|\mathbf{a}|,|\mathbf{b}|\} = -\mathbf{b} \quad \textrm{and} \quad L:=\mathbf{b}-\mathbf{a}\leq 0.
$$
Combining \eqref{eqEta} with the assumption that $3 \leq \mathbf{a}< -\mathbf{b}$ we get:
$$
\|\eta\|_{L^\infty}\leq \frac{3}{2 \alpha}
(1-\mathbf{b})\leq -\frac{2}{\alpha}\mathbf{b}.
$$
This, together with the bound $\|\partial_sh_s\|_{L^\infty}\leq \frac{\alpha}{3}$ and \eqref{inqEnergy}, gives us a contradiction:
\begin{gather*}
\mathbf{b} \geq \mathbf{a}- \|\eta\|_{L^\infty}\|\partial_sh_s\|_{L^\infty}\geq \mathbf{a} + \frac{2}{3}\mathbf{b}\\
0 \geq \frac{1}{3}\mathbf{b} \geq \mathbf{a}\geq 3.
\end{gather*}
\end{proof}

\subsection{The $L^2$-bounds}
In the following lemma we will show that for every pair $\mathbf{a},\mathbf{b}\in \R$ the corresponding space $\mathcal{M}^\Gamma(\mathbf{a},\mathbf{b})$ is bounded in the $L^2\times\R$-norm:

\begin{lem}\label{lem:L2bounds}
Consider a setting as in Theorem \ref{thm:FloerBounds}. Fix $\mathbf{a}, \mathbf{b}\in \R$, let $\alpha>0$ and $\beta\geq 0$ be as in Remark \ref{rem:HomConst}.
Let $\mathfrak{a,e}>0$ be the constants as in Lemma \ref{lem:ActBound},
such that for every $u\in \mathcal{M}^\Gamma(\mathbf{a},\mathbf{b})$ we have
$$
\left\|\mathcal{A}^{H_0-h_s}( u)\right\|_{L^{\infty}(\R)} \leq\mathfrak{a},\qquad \textrm{and}\qquad \left\|\nabla^{J_s} \mathcal{A}^{H_0-h_s}(u)\right\|_{L^{2}(\mathbb{R}\times [0,1])}^{2} \leq \mathfrak{e}.
$$
For $\varepsilon>0$ denote
\begin{align*}
\mathfrak{q} & := \min\{|q_0|, |q_1|\}+\frac{2}{\alpha}\left(\varepsilon(3\alpha+2\beta) +\mathfrak{a}\sqrt{2(\alpha+\beta)}\right) +\sqrt{\frac{2\beta}{\alpha}\left(\mathfrak{a}+\varepsilon \sqrt{2(\alpha+\beta)}\right)},
\end{align*}
Then for every $u=(v,\eta)\in \mathcal{M}^\Gamma(\mathbf{a},\mathbf{b})$, $v=(q,p)$ we have
\begin{align*}
\|q(s)\|_{L^2} & \leq \mathfrak{q}+\frac{\mathfrak{e}\|J\|_{\infty}}{\varepsilon},\\
\|p(s)\|_{L^2} & \leq 2|a|\mathfrak{q}+2\sqrt{\varepsilon+\beta}+\frac{\mathfrak{e}\|J\|_{\infty}}{\varepsilon},\\
\|v(s)\|_{L^2} & \leq \mathfrak{q}(1+2|a|)+2\sqrt{\varepsilon+\beta}+\frac{\mathfrak{e}\|J\|_{\infty}}{\varepsilon}.
\end{align*}
\end{lem}
\begin{proof}
On one hand, by Remark \ref{rem:HomConst} we can consider uniform constants $\alpha>0$ and $\beta\geq 0$ for all Hamiltonians $h_s$ in our homotopy. Consequently, the estimates in Proposition \ref{prop:Infinitesimal} can be applied to all the elements of $\bigcup_{s\in[0,1]}\mathcal{B}(H_a-h_s,\varepsilon)$. On the other hand, by Lemma \ref{lem:ActBound} we have uniform bounds on the energy and action for the Floer trajectories. Therefore, by Remark \ref{rem:BGammaBound} we obtain that the set $\mathcal{B}^\Gamma(\mathfrak{a},\varepsilon)$ is uniformly bounded in the $L^2\times \R$-norm. The rest of the proof follows from the arguments presented in the proof of \cite[Prop. 3.13]{Wisniewska2024}, thus it will not be presented in full length here, 
but we encourage a curious reader to see the details in \cite{Wisniewska2024}.
\end{proof}

\subsection{The maximum principle and the $L^\infty$-bounds}

\noindent \textit{Proof of Theorem  \ref{thm:FloerBounds}:}
Lemma \ref{lem:ActBound} we have uniform bounds on the energy and action for the Floer trajectories in $\mathcal{M}^\Gamma(\mathbf{a},\mathbf{b})$. Therefore, by Proposition \ref{prop:Infinitesimal} and Remark \ref{rem:BGammaBound} we know that the set $\mathcal{B}^\Gamma(\mathfrak{a},\varepsilon)$ is bounded in the $L^\infty\times\R$-norm. Therefore, if we fix $\varepsilon>0$, then there exists a compact set $K_\varepsilon\subseteq T^*\R^2$, such that
\begin{enumerate}[label=\roman*)]
\item for all $s\in \R$ and $x\in T^*\R^2\setminus K_\varepsilon$, we have $J_s(x,\eta) \equiv \mathbb{J}$ and $h_s(x) = \mathbf{c}_s \in \R_+$, where $s\mapsto\mathbf{c}_s$ is a smooth positive function with $\supp \partial_s\mathbf{c}_s \subseteq [0,1]$,
\item for $\mathfrak{a}>0$ as in Lemma \ref{lem:ActBound} and all $(v,\eta)\in \mathcal{B}^\Gamma(\mathfrak{a},\varepsilon)$, we have $v([0,1])\subseteq K_\varepsilon$.
\end{enumerate}

Now if we fix $u=(v,\eta) \in \mathcal{M}^\Gamma(\mathbf{a},\mathbf{b})$, then by Lemma \ref{lem:ActBound} the convergence of the integral $\|\nabla \A^{H_a-h_s} \circ u\|_{L^2\times\R}\leq \|J\|_\infty \mathfrak{e}$ assures that $\lim_{s\pm \infty}u(s) \in \mathcal{B}^\Gamma(\mathfrak{a},\varepsilon)$. Therefore, to obtain uniform $L^\infty\times \R$-bounds on $\mathcal{M}^\Gamma(\mathbf{a},\mathbf{b})$ it suffices to show that for any connected component $\Omega \subseteq \R\times[0,1]\setminus v^{-1}(K_\varepsilon)$ the set $v(\Omega)$ is bounded in $T^*\R^2$ and the bounds do not depend on the choice of $u$ or $\Omega$.
The standard method to obtain the bounds on $v(\Omega)$ is to apply the Alexandrov's maximum principle (see \cite[Sec. 3.4]{Wisniewska2024} for details). Note that at $(s,t)\in \Omega$ the Floer equations become
\begin{align*}
\partial_sv(s,t) & = \mathbb{J}(\partial_tv-X_{H_a}(v))(s,t),\\
\partial_s \eta (s) & = -\int_0^1 (H_a -h_s) \circ v(s,t)dt.
\end{align*}
 
Observe that the setting here does not differ substantially from the one presented in \cite[Sec. 3.5]{Wisniewska2024}. In fact, it differs in two aspects:
\begin{enumerate}[label=\roman*)]
\item we consider Hamiltonian $H_1$ instead of $H_a$ for an arbitrary $a\in \R$,
\item the homotopy of smooth Hamiltonian perturbations $\{h_s\}_{s\in \R}$ consists of elements of the set $\mathcal{H}$ as in \eqref{DefH}, which additionally satisfy $h_s >0$ and $h_s-dh_s(p\partial_sp)>0$.
\end{enumerate}

In \cite[Sec. 3.5]{Wisniewska2024} to prove the uniform $L^\infty\times\R$-bounds on $\mathcal{M}^\Gamma(\mathbf{a},\mathbf{b})$ one first proves Lemmas 3.18 to 3.24.
However, if we were to choose $H_a$ for an arbitrary $a\in \R$ instead of $H_1$, then the calculations in the proofs of Lemmas 3.18 to 3.24 in \cite{Wisniewska2024} would be exactly the same, with the only difference of $|a|$ appearing in some of the bounds. Moreover, in the proofs of Lemmas 3.21 to 3.24 in \cite{Wisniewska2024} we do not use the assumption ii)
, but only that $\sup_{s\in \R} \| h_s\|_{W^{1,\infty}}<+\infty$, which is satisfied also in the setting of Theorem \ref{thm:FloerBounds}. Therefore, we think that redoing the calculations and including them here would not add anything substantial to the proof as they would follow exactly the same pattern as the ones already presented in the proofs of Lemmas 3.18 to 3.24 in \cite{Wisniewska2024}. 

Based on the proof of maximum principle presented in \cite[Sec. 3.5]{Wisniewska2024}, we conclude that $\mathcal{M}^\Gamma(\mathbf{a},\mathbf{b})$ is bounded in the $L^\infty\times \R$-norm.

\hfill $\square$

\section{Positive LRFH}
\label{sec:posLRFH}
The action functional $\A^{H}_{q_0,q_1}$ provides an $\R$-filtration on $CF_*(\A^{H}_{q_0,q_1})$ by
$$
CF_*^{\leq a}\left(\A^{H}_{q_0,q_1}\right)\coloneqq \left\lbrace {\textstyle \sum_{x\in S}x}\in CF_*\left(\A^{H}_{q_0,q_1}\right) \left|\ \sup_{x\in S}\A^{H}_{q_0,q_1}(x) \leq a\right.\right\rbrace.
$$
Since Floer trajectories are defined by the $L^2$-gradient of the action functional, the boundary operator does not increase the action, i.e.\
\begin{equation}\label{filtration}
\partial\left( CF^{\leq a}_{*+1}\left(\A^{H}_{q_0,q_1}\right)\right)\subseteq CF^{\leq a}_*\left(\A^{H}_{q_0,q_1}\right).
\end{equation}
The {\em positive Rabinowitz Floer homology} $\LRFH_*^+(\A^{H}_{q_0,q_1})$ is the homology of the quotient complex
$$
  CF_*^+\left(\A^{H}_{q_0,q_1}\right) \coloneqq CF_*\left(\A^{H}_{q_0,q_1}\right)\Big/CF_*^{\leq 0}\left(\A^{H}_{q_0,q_1}\right),
$$
with boundary operator $\partial^{\scriptscriptstyle +}$ induced by $\partial$ on the quotient.

The main aim of this section is to prove Theorem \ref{thm:posLRFH}. In order to show that the positive Lagrangian Rabinowitz Floer homology is not only well defined for all Hamiltonians of the form $H_a-h$ with $H_a$ as in \eqref{DefHa} and $h \in \mathcal{H}$ with $\mathcal{H}$ as in \eqref{DefH}, but also independent of the auxiliary choices and invariant under compact perturbations, we will use the following general result from \cite{Wisniewska2024}:

\begin{prop}\label{prop:DefLRFH+}
Consider the setting as in Theorem \ref{thm:DefLRFH} with sets $K\subseteq\mathcal{V}\subseteq T^*Q$ and $\mathcal{H}\subseteq \{C^\infty(T^*Q)\ |\ dh\in C_0^\infty(K)\}$. Assume $q_0\neq q_1$. Let $\mathcal{O}\subseteq\mathcal{H}$ be an open neighbourhood of $0$ such that for every pair $h_0,h_1 \in \mathcal{O}$ there exists a homotopy $\Gamma:= \{(H+h_s, J_s)\}_{s\in \R}$ satisfying Condition 3 of Theorem \ref{thm:DefLRFH}, and such that for every $x\in \Crit^+\A^{H+h_0}_{q_0,q_1}$ and every $y\in \Crit\A^{H+h_1}_{q_0,q_1}$ for which $\F_\Gamma(x,y)\neq \emptyset$ we have $\A^{H+h_1}_{q_0,q_1}(y)>0$.

Then for every $h\in\mathcal{O}$ its positive Lagrangian Rabinowitz Floer homology is well defined, and for every pair $h_0,h_1 \in \mathcal{O}$, $\LRFH_*^+(\A^{H+h_0}_{q_0,q_1})$ is isomorphic to $\LRFH_*^+(\A^{H+h_1}_{q_0,q_1})$.
\end{prop}

We already have shown in Lemma \ref{lem:nonEmpty}, Lemma \ref{lem:Chord}, Lemma \ref{lem:Novikov} and Theorem \ref{thm:FloerBounds} that Hamiltonians of the form $H_a-h$ with $H_a$ as in \eqref{DefHa} and $h \in \mathcal{H}$ with $\mathcal{H}$ as in \eqref{DefH} satisfy assumptions of Theorem \ref{thm:DefLRFH}. In the following Lemmas we will show that $H_a-h$ satisfy all the assumptions of Proposition \ref{prop:DefLRFH+}:

\begin{lem}
Let $H_a$ be a Hamiltonian as in \eqref{DefHa} and let $\mathcal{H}$ be a set of perturbations as in \eqref{DefH}. Fix $h_0\in \mathcal{H}$ and let $K \subseteq T^*\R^2$ be a compact subset and $\mathcal{H}_{h_0}(K)\subseteq C_0^\infty(K)$ an open neighbourhood of $0$, such that for every $h\in \mathcal{H}_{h_0}(K)$ we have $h_0+h\in \mathcal{H}$ and for every $(v, \eta)\in \Crit \A^{H_a-h_0-h}$ we have $v([0,1])\subseteq K$. Then there exists an open neighbourhood $\mathcal{O}_{h_0}(K)\subseteq \mathcal{H}_{h_0}(K)$, such that
$$
\inf \left\lbrace \mathcal{A}^{H_a-h_0-h}(v,\eta)\ \Big|\ h\in \mathcal{O}_{h_0}(K), \quad (v,\eta)\in \Crit^+ \mathcal{A}^{H_a-h_0-h} \right\rbrace >0.
$$
\end{lem}

\begin{proof}
We will argue by contradiction. Suppose we have a sequence,\linebreak $h_n\in \mathcal{H}_{h_0}(K)$ and $(v_n,\eta_n)\in \Crit^+\A^{H_a-h_0-h_n}$, such that $\lim_{n\to \infty}h_n=0$ and\linebreak $\lim_{n\to \infty}\A^{H_a-h_0-h_n}(v_n,\eta_n)= 0$. Since $h_0, h_0+h_n\in \mathcal{H}$ and $\lim_{n\to \infty}h_n=0$ we can find $\alpha>0$, such that every $h_0+h_n$ satisfies \eqref{DefAlpha} with $\alpha$. By \eqref{dA^H(Y)} and \eqref{d(H_a-h)} we get
$$
\A^{H_a-h_0-h_n}(v_n, \eta_n) = \eta_n\int\left(h_0+h_n-d(h_0+h_n)(p\partial_p)+\frac{1}{2} |p \circ v_n|^2\right)\geq \alpha \eta_n.
$$
Thus $\eta_n>0$ and $\lim_{n\to \infty}\eta_n=0$. We have $\lim_{n\to \infty}h_n=0$ and $\supp h_n \subseteq K$, hence $\limsup \|X_{H_a-h_0-h_n}\|_{L^\infty(K)}<+\infty$.
Since $v_n([0,1])\subseteq K$, we obtain
\begin{align*}
\lim_{n\to \infty}\|\partial_s v_n\| & =\lim_{n\to \infty}\eta_n\|X_{H_a-h_0-h_n} (v_n)\|\\
& \leq \lim_{n\to \infty}\eta_n \limsup_{n\to \infty} \|X_{H_a-h_0-h_n}\|_{L^\infty(K)} =0.
\end{align*}
By Ascola-Arzeli this implies that there exists $v_0\in \mathscr{H}_{q_0,q_1}$, such that\linebreak $\lim_{n\to\infty}(v_n,\eta_n)=(v_0,0)$ in $W^{1,2}$-topology. By continuity of the Rabinowitz action functional $(v_0,0)\in \Crit \A^{H_a-h_0}$. But for $q_0 \neq q_1$ there does not exist an element of 
$ \mathscr{H}_{q_0,q_1}$, such that $\partial_tv_0=0$, which brings us to a contradiction.
\end{proof}

\begin{lem}\label{lem:positive}
Let $H_a$ be a Hamiltonian as in \eqref{DefHa} and let $\mathcal{H}$ be a set of perturbations as in \eqref{DefH}. Fix $h_0\in \mathcal{H}$ and let $K \subseteq T^*\R^2$ be a compact subset and $\mathcal{O}_{h_0}(K)\subseteq C_0^\infty(K)$ an open neighbourhood of $0$, such that
\begin{enumerate}[label=\roman*)]
\item $\{h_0+h\ |\ h\in \mathcal{O}_{h_0}(K)\}\subseteq \mathcal{H}$;
\item the critical set of the Rabinowitz action functional is continuously compact in $(K, \mathcal{O}_{h_0}(K)$;
\item $\alpha:=\inf \left\lbrace h-dh(p\partial_p)+\frac{1}{2}|p|^2\ \Big|\ h\in \mathcal{O}_{h_0}(K) \right\rbrace >0$;
\item $\beta:= \sup\left\lbrace 
\max\left\lbrace \sup h, \sup h-dh(p\partial_p), \sup pdp(X_h), -\inf qdq(X_h)\right\rbrace\ |\ h-h_0\in \mathcal{O}_{h_0}(K)\right\rbrace < +\infty$;
\item $\delta(h_0):=\inf \left\lbrace \mathcal{A}^{H_a-h_0-h}(v,\eta)\ \Big|\ h\in \mathcal{O}_{h_0}(K), \quad (v,\eta)\in \Crit^+ \mathcal{A}^{H_a-h_0-h} \right\rbrace$.
\end{enumerate}
Fix $h\in \mathcal{O}_{h_0}(K)$ and denote $h_1:= h_0+h$. Let $\mathcal{V}\subseteq T^*\R^2$ be an open, but precompact subset, such that $K\subseteq \mathcal{V}$.
 Let $\Gamma:=\{(h_s, J_s)\}_{s\in \R}$ be a smooth homotopy, such that $h_s-h_0 \in \mathcal{O}_{h_0}(K)$ and $2$-parameter families of almost complex structures $J_s \in C^\infty([0,1]\times \R,\mathcal{J}(\mathcal{V},\mathcal{J}))$ constant in $s$ outside $[0,1]$ satisfying \eqref{CondGamma} and such that
\begin{equation}\label{Gamma2}
\|\partial_{s}h_{s}\|_{L^{\infty}}
\max\left\lbrace 1+\frac{1}{\delta(h_0)}, 1+\alpha \sqrt{2(\alpha+\beta)}\|J\|_{L^\infty}^2\right\rbrace
 \leq \frac{\alpha}{3}.
\end{equation}
Then for every $x\in \Crit^+\A^{H_a-h_0}_{q_0,q_1}$ and every $y\in \Crit\A^{H_a-h_1}_{q_0,q_1}$, such that\linebreak $\F^\Gamma(x,y)\neq \emptyset$ we have $\A^{H_a-h_1}_{q_0,q_1}(y)>0$.
\end{lem}

\begin{proof}
First observe that our homotopy $\Gamma$ satisfies assumptions of Theorem \ref{thm:FloerBounds}, hence we can apply the estimates from the proof of Lemma \ref{lem:ActBound}.

Fix $x\in \Crit^+\A^{H_a-h_0}_{q_0,q_1}$ and $y\in \Crit\A^{H_a-h_1}_{q_0,q_1}$ and abbreviate:
$$
\mathbf{a}= \A^{H_a-h_0}_{q_0,q_1}(x) \qquad\textrm{and}\qquad \mathbf{b}=\A^{H_a-h_1}_{q_0,q_1}(y).
$$
By assumption $\mathbf{a}\geq \delta(h_0)$. Let $u=(v,\eta)\in \F_\Gamma(x,y)$.

First, suppose $|\mathbf{b}|\leq \mathbf{a}$. Then $\mathbf{b}-\mathbf{a} \leq 0$ and \eqref{eqEta} gives
\begin{align}
\|\eta\|_{L^\infty} & \leq \frac{3}{2}\left( \frac{1}{\alpha}\left(\max\{|\mathbf{a}|,|\mathbf{b}|\}+1\right)+(\mathbf{b}-\mathbf{a}) \sqrt{2(\alpha+\beta)}\|J\|_{L^{\infty}}^2\right)\nonumber\\
& \leq \frac{3}{2\alpha}(\mathbf{a}+1) \leq \frac{3\mathbf{a}}{2\alpha}\left(1+\frac{1}{\delta(h_0)} \right). \label{etaBDdelta}
\end{align}
On the other hand, by equation \eqref{inqEnergy} we have
$$
\left\|\nabla^{J_{s}} \mathcal{A}^{H_0-h_{s}}(u)\right\|_{L^{2}(\mathbb{R}\times [0,1])}^{2}
 \leq \|J\|_{L^{\infty}} (\mathbf{b}-\mathbf{a} + \|\eta\|_{L^{\infty}}\|\partial_{s}h_{s}\|_{L^{\infty}}).
$$
Combined with \eqref{etaBDdelta} and \eqref{Gamma2}, we obtain the following estimate:
$$
\mathbf{b}\geq \mathbf{a}-\|\eta\|_{L^{\infty}}\|\partial_{s}h_{s}\|_{L^{\infty}} \geq \mathbf{a}\left(1-\frac{3}{2\alpha}\left(1+\frac{1}{\delta(h_0)} \right)\|\partial_{s}h_{s}\|_{L^{\infty}} \right)\geq \frac{1}{2}\mathbf{a}.
$$
Since $h_1-h_0 \in \mathcal{O}_{h_0}(K)$, we conclude that $\mathbf{b}\geq \delta(h_0)>0$.

Now suppose $\mathbf{b}< -\mathbf{a}\leq - \delta(h_0)<0$. Then $\mathbf{b}-\mathbf{a}<0$ and by \eqref{eqEta} we obtain
$$
\|\eta\|_{L^\infty} \leq \frac{3}{2\alpha}\left(1-\mathbf{b}\right)< -\frac{3\mathbf{b}}{2\alpha}\left(1+\frac{1}{\delta(h_0)} \right).
$$
Combined with equations \eqref{inqEnergy} and \eqref{Gamma2} we get the following contradiction:
$$
\mathbf{a} \leq \mathbf{b}+ \|\eta\|_{L^{\infty}}\|\partial_{s}h_{s}\|_{L^{\infty}}\leq \mathbf{b}\left(1 -\frac{3}{2\alpha}\left(1+\frac{1}{\delta(h_0)} \right)\|\partial_{s}h_{s}\|_{L^{\infty}}\right)\leq \frac{1}{2}\mathbf{b}< -\frac{1}{2}\mathbf{a}.
$$
\end{proof}

Having proven Lemma \ref{lem:positive} we can apply Proposition \ref{prop:DefLRFH+} to conclude that the positive Lagrangian Rabinowitz Floer homology is well defined for all perturbations from the set $\mathcal{H}$ and in fact isomorphic for any two perturbations.
What is left to conclude the proof of Theorem \ref{thm:posLRFH} is to calculate the homology:

\begin{lem}\label{lem:calculation}
Let $H_a$ be a Hamiltonian as in \eqref{DefHa}. Fix $q_0,q_1\in \R^2$, $q_0\neq q_1$. Then for every $c>\frac{1}{2}a^2|q_0||q_1|$ we have
$$
\LRFH_*^+\left(\A^{H_a-c}_{q_0,q_1} \right)=\begin{cases}
\mathbb{Z}_2 & \textrm{for}\quad *=\frac{1}{2},\\
0 & \textrm{otherwise}.
\end{cases}
$$
\end{lem}

\begin{proof}
Define $\varphi_a(q,p):=\left(\frac{1}{\sqrt{a}}q, \sqrt{a}p \right)$. Then $\varphi_a$ is a symplectomorphism of $(T^*\R^2, dq\wedge dp)$ and we have
\begin{align*}
H_a \circ \varphi_a (q,p) & = \frac{a}{2}|p|^2+a(p_1q_2-p_2q_1)= a H_1(q,p),\\
\omega(\ \cdot\ , X_{H_a}) & = dH_a  = a\cdot d\left(H_1 \circ \varphi_a^{-1}\right) \\
& = a\cdot \omega(D\varphi_a^{-1}\ \cdot\ , X_{H_1})= \omega(\ \cdot\ , aD\varphi_a (X_{H_1}))
\end{align*}
Thus $X_{H_a}=aD\varphi_a (X_{H_1})$ and $H^{-1}_{a}(c)= \varphi_a\left(H_1^{-1}(\frac{c}{a})\right)$, which gives us a $1$-to-$1$ correspondence
\begin{equation}\label{iff}
(v,\eta) \in \Crit\A^{H_a-c}_{q_0,q_1}\quad \iff \quad (\varphi_a^{-1}\circ v, a\eta)\in \Crit\A^{H_1-\frac{c}{a}}_{\sqrt{a}q_0,\sqrt{a} q_1}.
\end{equation}
On the other hand, by \cite[Lem. 4.3]{Wisniewska2024} we know that for $c > \frac{1}{2}a^2|q_0||q_1|$ we have $\#\Crit^+\A^{H_1-\frac{c}{a}}_{\sqrt{a}q_0,\sqrt{a} q_1}= \#\Crit^-\A^{H_1-\frac{c}{a}}_{\sqrt{a}q_0,\sqrt{a} q_1}=1$. Consequently, for $c > \frac{a^2}{2}|q_0||q_1|$ we have $\# \Crit^+\A^{H_a-c}_{q_0,q_1}=\# \Crit^-\A^{H_a-c}_{q_0,q_1}=1$. In particular, $\LRFH_*^+\left(\A^{H_a-c}_{q_0,q_1} \right)$ is well defined, as the complex has only one generator and the boundary operator is $0$.

Moreover, by the proof of \cite[Thm. 1.2]{Wisniewska2024} we know that the Maslov index of the only chord of $\Crit^+\A^{H_1-\frac{c}{a}}_{\sqrt{a}q_0,\sqrt{a} q_1}$ is equal to $\frac{1}{2}$. Since a symplectomorphism preserves the Maslov index, we infer that the Maslov index of the unique element of $\Crit^+\A^{H_a-c}_{q_0,q_1}$ is also $\frac{1}{2}$, which concludes the proof.
\end{proof}
\begin{rem}
By the correspondence given by \eqref{iff} and the result for $H_1$\linebreak in \cite[Prop. 4.9]{Wisniewska2024} we can deduce that for all $a>0$ and all $q_0,q_1\in \R^2$, $q_0\neq q_1$ we have
$$
\lim_{c \searrow 0}\# \Crit^+\A^{H_a-c}_{q_0,q_1}=+\infty.
$$
\end{rem}

We finish this section with the proof of Theorem \ref{thm:posLRFH}:

\vspace*{.5cm}
\noindent\textit{Proof of Theorem \ref{thm:posLRFH}:}
Fix $q_0, q_1 \in \R^2$, $q_0\neq q_1$.
By Lemma \ref{lem:nonEmpty}, Lemma \ref{lem:Chord}, Lemma \ref{lem:Novikov}, Theorem \ref{thm:FloerBounds} and Lemma \ref{lem:positive} we know that
for any $H_a$ as in \eqref{DefHa} and any $h_0\in \mathcal{H}$ as in \eqref{DefH}, there exists an open neighbourhood $\mathcal{O}_{h_0}$ of $0$ in $\{h \in C^\infty(T^*\R^2)\ |\ dh\in C_c^\infty(T^*\R^2) \}$ for which all the assumptions of Proposition \ref{prop:DefLRFH+} are satisfied. Therefore, for any $H_a$ and any $h\in \mathcal{H}$ the associated positive Lagrangian Rabinowitz Floer homology is well defined. Moreover, for any two $h_0,h_1\in \mathcal{H}$ the homology $\LRFH_*^+(\A^{H_a-h_0}_{q_0,q_1})$ is isomorphic to $\LRFH_*^+(\A^{H_a-h_1}_{q_0,q_1})$. Consequently, by Lemma \ref{lem:calculation} we have 
$$
\LRFH_*^+\left(\A^{H_a-h}_{q_0,q_1} \right)=\begin{cases}
\mathbb{Z}_2 & \textrm{for}\quad *=\frac{1}{2},\\
0 & \textrm{otherwise}.
\end{cases}
$$
for any $h \in \mathcal{H}$.

\hfill $\square$

\section{The regularized rotating Kepler problem}
\label{sec:RegHam}
The main aim of this section is to prove Theorem \ref{thm:Kepler}. We will start by analysing basic properties of the regularized Hamiltonian \eqref{regHam}. In \eqref{DefHc} we will construct a function $h \in C^\infty(T^*\R^2)$ and then we will show that $h \in \mathscr{H}$ for the set $\mathscr{H}$ defined as in \eqref{DefH}. Finally, we will show that such constructed $h$ satisfies the assertions of Theorem \ref{thm:Kepler}. 

\subsection{Properties of the regularized Hamiltonian}
In this subsection we will analyse the geometric relations between the regularized Hamiltonian $K_c$ as in \eqref{regHam} and the quadratic Hamiltonian $H_a$ as in \eqref{DefHa}.

First, we will show that the restriction of the zero-level set of $K_c$ to $T^*B(R)$ can be included in a sublevel set of $H_a$:

\begin{lem}\label{lem:alpha}
For $R>0$ set $a:=2R^2$ and let $H_a$ be the corresponding Hamiltonian as in \eqref{DefHa}. Then there exists $b>0$, such that 
$$
K_c^{-1}(0)\cap \{r\leq R\}\subseteq H_a^{-1}((-\infty, b]).
$$
\end{lem}
\begin{proof}
Using \eqref{DefHa} for $(r,\theta,p_r,p_\theta)\in H_a^{-1}((-\infty,b])$ we can estimate
\begin{align}
\frac{1}{2}p_r^2+ \frac{p_\theta^2}{2r^2}+ap_\theta & \leq b,\nonumber\\
p_r^2+ \left(\frac{p_\theta}{r}+ar \right)^2 & \leq a^2r^2+2b.\label{inq2}
\end{align}
On the other hand, for $(r,\theta,p_r,p_\theta)\in K_c^{-1}(0)$ by \eqref{regHam} we have
\begin{align}
\frac{1}{2}p_r^2+ \frac{p_\theta^2}{2r^2}+2 r^2p_\theta & = 4(cr^2+1),\label{eqKc}\\
p_r^2 +\left(\frac{p_\theta}{r}+2r^3 \right)^2 & = 4(r^6+2(cr^2+1)).\nonumber 
\end{align}
Therefore it suffices to choose $b>0$, such that for all $r\leq R$ we have
$$
\sqrt{a^2r^2+2b} \geq 2\sqrt{r^6+2(cr^2+1)}+r(a-2r^2)
$$
This means it suffices for $b$ to satisfy the following condition:
\begin{equation}\label{b}
b \geq \frac{1}{2}\max_{r \leq R}\left(\left(2\sqrt{r^6+2(cr^2+1)}+r(a-2r^2) \right)^2-a^2r^2\right).
\end{equation}
Since the function on the right hand side is smooth and smooth functions on a closed interval are bounded, we can always find $b>0$ that satisfies the condition above.
\end{proof}

In the following lemma we will prove some property that late on will be useful in understanding how the energy value influences the dynamics on the energy hypersurface.
\begin{lem}\label{lem:pTheta}
Let $K_c$ be the Hamiltonian defined in \eqref{regHam}.
Fix $(r,\theta)\in \R_+\times S^1$. Then for all $c\geq \sqrt{2}r$ and all $(p_r,p_\theta)\in T_r^*\R_+\times T_\theta^*S^1$, such that $K_c(r,\theta,p_r,p_\theta)=0$ we have $p_\theta \leq 2c$.
\end{lem}

\begin{proof}
Denote $e:= 2c - p_\theta$. Then for $(r,\theta, p_r, p_\theta)\in K_c^{-1}(0)$ by \eqref{eqKc} we have
$$
0 \geq \frac{(2c-e)^2}{2r^2}- 2 r^2 e-4 = \frac{1}{r^2}\left( \frac{1}{2}e^2-2(c+r^4)e+2c^2-4r^2\right).
$$
From this quadratic inequality, we can infer that
$$
 e \geq 2(c+r^4)-\sqrt{4(c+r^4)^2+4(2r^2-c^2)}.
$$
Consequently, for $c \geq \sqrt{2}r$ we have $e\geq 0$.
\end{proof}


\subsection{Constructing the perturbation}

Let $\chi:\R \to [0,1]$ be a smooth, not increasing function on $\R$, such that\linebreak $\inf \chi' \geq -2$ and 
\begin{equation}\label{DefChi}
\chi(x):=\begin{cases}
1 & \textrm{for}\quad x\leq 0,\\
0 & \textrm{for}\quad x\geq 1.
\end{cases}
\end{equation}

For constants $a, b, c, d, R_1, R_2>0$ with $R_2>R_1$ we define the following smooth functions on $T^*\R^2$:
\begin{align}
\chi_0(q,p)  & :=\chi \left(\frac{|q|-R_1}{R_2-R_1}\right),\label{DefChi0}\\
\chi_1 (q,p) & := \chi \left(\frac{1}{d}(H_a(q,p)- b)\right),\label{DefChi1}\\
 \varphi(x) & :=\chi_0(x) \chi_1(x)\label{DefVarphi}
\end{align}
\begin{equation}
\begin{aligned}
h (x) & := \varphi \left(a -2|q|^2\right)\left(p_1q_2-p_2q_1-2c\right)+2ac +4\\
& = \varphi \left(a -2r^2\right)(p_\theta-2c)+2ac +4
\label{DefHc}
\end{aligned}
\end{equation}
In the next two lemmas we will prove that $h$ as in \eqref{DefHc} is an element of the set $\mathcal{H}$ defined in \eqref{DefH}:

\begin{lem}\label{lem:dh_compact}
The support of $dh$ as in \eqref{DefHc} is compact.
\end{lem}
\begin{proof}
By definition
$$
\supp dh \subseteq \supp \varphi = \supp \chi_0 \cap \supp \chi_1 = \left\lbrace r \leq R_2 \right\rbrace \cap H_a^{-1}((-\infty, b+d]).
$$
Using \eqref{inq2} we can deduce that 
\begin{align*}
&\textrm{for}& (r, \theta, p_r, p_\theta) &\in T_{(r,\theta)}(R_+\times S^1) \cap H_a^{-1}((-\infty, b+d]),\\
&\textrm{we have} &|p_r| & \leq \sqrt{a^2r^2+2(b+d)},\\
&\textrm{and}&p_\theta & \in [-ar^2-r\sqrt{a^2r^2+2(b+d)}, -ar^2+ r\sqrt{a^2r^2+2(b+d)}].
\end{align*}
Consequently, 
\begin{align*}
&\textrm{for}& (r, \theta, p_r, p_\theta) &\in \left\lbrace r \leq R_2 \right\rbrace \cap H_a^{-1}((-\infty, b+d]),\\
&\textrm{we have} &|p_r| & \leq \sqrt{a^2R_2^2+2(b+d)},\\
&\textrm{and}&p_\theta & \in \left[-R_2 \left(aR_2+\sqrt{a^2R_2^2+2(b+d)}\right),  R_2\sqrt{2(b+d)}\right].
\end{align*}
\end{proof}
\begin{lem}\label{lem:ContactType}
Fix $c, R_1, R_2>0$ with $R_2>R_1$ and
let $H_a$ be as in \eqref{DefHa} with $a=2R_2^2$. Choose a constant $b \geq 2ac+ 4$ as in Lemma \ref{lem:alpha}. Then there exists $d>0$ big enough for the corresponding $h$ defined as in \eqref{DefHc} to satisfy
$$
h-dh(p\partial_p)+\frac{1}{2}|p|^2 >0.
$$
\end{lem}

\begin{proof}
Using \eqref{DefHc} we can calculate
\begin{align*}
dh(p\partial_p) & = \varphi(a-2r^2)p_\theta+(a-2r^2)(p_\theta-c)d\varphi(p\partial_p),\\
h-dh(p\partial_p) & = 4(1+\varphi cr^2) +2ac(1-\varphi) -(a-2r^2)(p_\theta-c)d\varphi(p\partial_p).
\end{align*}
Therefore we can infer that $h-dh(p\partial_p)\geq 4$ on $T^*\R^2 \setminus \supp d\varphi$. In fact, since $d\varphi(p\partial_p) = \chi_0 d\chi_1(p\partial_p)$ we have $h-dh(p\partial_p)\geq 4$ on
$$
T^*\R^2 \setminus \left(\supp d\chi_1\cap \supp \chi_0\right)= T^*\R^2 \setminus \left( \{ r\leq R_2\}\cap H_a^{-1}([b, b+d])\right).
$$
Now using the definition of $\chi_1$ we can calculate
$$
d\chi_1(p\partial_p)=\frac{1}{d} \chi' dH_a(p\partial_p)=\frac{1}{d}\chi' \left(\frac{1}{2}|p|^2+H_a \right).
$$
By assumption $b, d>0$, $\chi_0 \geq 0$ and $\chi'\leq 0$, hence we obtain $d\varphi(p\partial_p)\leq 0$ and therefore we can infer that $h-dh(p\partial_p)\geq 4$ on
$$
T^*\R^2 \setminus \left( \{ r\leq R_2\}\cap\left\lbrace p_\theta \leq c\right\rbrace\cap H_a^{-1}([b, b+d])\right).
$$

Using \eqref{inq2} we can parametrize $H_a^{-1}([b, b+d])$ as an image of the the following map from $\R_+\times S^1\times S^1\times [0,1]$ to $T^*(\R_+\times S^1)$:
$$
(r,\theta, \phi, t)\mapsto \left(r,\theta, \sqrt{a^2r^2+2(b+td)}\cos \phi, r\left(\sqrt{a^2r^2+2(b+td)}\sin \phi - ar\right)\right).
$$

Recall that by assumption $0\geq \chi'\geq -2$.
Hence on $\{ r\leq R_2\}\cap\left\lbrace p_\theta \leq c\right\rbrace\cap H_a^{-1}([b, b+d])$ we can calculate:
\begin{align*}
h & -dh(p\partial_p) +\frac{1}{2}|p|^2 = \frac{1}{2}|p|^2+ 4(1+\varphi cr^2) +2ac(1-\varphi)\\ & +\frac{1}{d}\chi_0\chi'(a-2r^2)(c-p_\theta)\left(\frac{1}{2}|p|^2+H_a\right)\\
& \geq \frac{1}{2}|p|^2+ 4 -\frac{2}{d}(a-2r^2)(c-p_\theta)\left(\frac{1}{2}|p|^2+H_a\right)\\
& = H_a-ap_\theta+ 4 -\frac{2}{d}(a-2r^2)(ap_\theta^2-(2H_a-a)p_\theta+2cH_a)\\
& \geq  H_a-ap_\theta+ 4 -\frac{2}{d}(a-2r^2)(2ar^2(H_a-ap_\theta)-(2H_a-a)p_\theta+2cH_a)\\
& = H_a-ap_\theta+ 4 -\frac{2}{d}(a-2r^2)(2(ar^2+c)H_a-(2H_a+2a^2r^2-a)p_\theta)\\
& = H_a - \left(a-4(a-2r^2)\frac{H_a}{d} \right)p_\theta - 4(a-2r^2)(ar^2+c)\frac{H_a}{d}\\
&+4+2a(a-2r^2)(2ar^2-1)\frac{p_\theta}{d}.
\end{align*}
Recall, that by the parametrization of $H_a([b, b+d])$ we have
\begin{align*}
p_\theta & = r\left(\sqrt{a^2r^2+2(b+td)}\sin \phi - ar \right)\\
& \geq - r\left(\sqrt{a^2r^2+2(b+td)}+ ar \right)\\
& \geq -r \left( 2ar +\sqrt{2(b+d)}\right)
\end{align*}
Using this inequality and the substitution $H_a=b+td$ we obtain:
\begin{align*}
h & -dh(p\partial_p) +\frac{1}{2}|p|^2  \geq b+td - ap_\theta +4(a-2r^2)t p_\theta - 4(a-2r^2)(ar^2+c)t\\
&+4+2(a-2r^2)\left(a (2ar^2-1)+ 2b \right)\frac{p_\theta}{d}- 4(a-2r^2)(ar^2+c) \frac{b}{d}\\
& = b -ac +a(c-p_\theta)-4tr(a-2r^2)\left( 2ar +\sqrt{2(b+d)}\right)
- 4(a-2r^2)(ar^2+c)t\\
&+4+2(a-2r^2)\left(a (2ar^2-1)+ 2b \right)\frac{p_\theta}{d}- 4(a-2r^2)(ar^2+c) \frac{b}{d}\\
& \geq 8+ac +t\left(d-4(a-2r^2)\left(3ar^2+c+ r\sqrt{2(b+d)}\right) \right)\\
&+2(a-2r^2)\left(a (2ar^2-1)+ 2b \right)\frac{p_\theta}{d}- 4(a-2r^2)(ar^2+c) \frac{b}{d}.
\end{align*}
Now, observe that for $r\leq R_2=\sqrt{\frac{a}{2}}$ we have
\begin{align*}
d-4(a-2r^2)\left(3ar^2+c+ r\sqrt{(b+d)}\right) &\geq d-2a(3a^2+2c+2\sqrt{a(b+d)}),\\
\lim_{d\to +\infty}\left(d-2a(3a^2+2c+2\sqrt{a(b+d)}) \right) & = + \infty,\\
p_\theta = r\sqrt{a^2r^2+2(b+td)}\sin\phi-ar^2&,\\
\lim_{d\to +\infty}\frac{|p_\theta|}{d} (a-2r^2)|a (2ar^2-1)+ 2b| & =0,\\
\lim_{d\to +\infty}\frac{1}{d} b(a-2r^2)(ar^2+c) & =0.
\end{align*}
We conclude that for $d>0$ big enough, we have $h-dh(\partial_p)+\frac{1}{2}|p|^2>0$ on the whole $T^*\R^2$.
\end{proof}

By Lemma \ref{lem:dh_compact} and Lemma \ref{lem:ContactType} we know that if we fix $c, R_1, R_2>0$ with $R_2>R_1$ and $a:= 2R_2^2$ then there exist $b, d>0$ big enough, such that the corresponding $h$ defined as in \eqref{DefHc} belongs to the set $\mathcal{H}$ defined as in \eqref{DefH}. Such constructed $h$ is our candidate for Theorem \ref{thm:Kepler}. However, before we prove the Theorem, we will first prove the following two lemmas:

\begin{lem}\label{lem:inclusion}
Fix $c, R_1, R_2>0$ with $R_2>R_1$ and
let $H_a$ be as in \eqref{DefHa} with $a=2R_2^2$. Choose a constant $b \geq 2ac+ 4$ as in Lemma \ref{lem:alpha}. Then for the corresponding $h$ defined as in \eqref{DefHc} we have
$$
(H_a-h)^{-1}(0)\subseteq H_a^{-1}((-\infty, b]).
$$
\end{lem}

\begin{proof}

Consider $(r,\theta, p_r, p_\theta)\in (H_a-h)^{-1}(0)$. Then
\begin{align*}
\frac{1}{2}p_r^2+ \frac{p_\theta^2}{2r^2}+\left(2r^2\varphi + a(1-\varphi)\right)p_\theta & = 4(cr^2\varphi+1) +2ac(1-\varphi)\\
p_r^2+\left(\frac{p_\theta}{r}+ r\left(2r^2\varphi + a(1-\varphi)\right)\right)^2 & = r^2\left(2r^2\varphi + a(1-\varphi)\right)^2+ 8(cr^2\varphi+1) +4ac(1-\varphi)
\end{align*}
By definition
$$
\supp \varphi = \supp \chi_0 \cap \supp \chi_1 = \{ r\leq R_2\} \cap H_a^{-1}((-\infty, b+d]).
$$
Consequently, on $\{r\geq R_2\}$ we have $H_a-h = H_a -2ac-4$,
which proves that 
$$
(H_a-h)^{-1}(0)\cap \{r\geq R_2\} \subseteq H_a^{-1}(2ac+4)\subseteq H_a^{-1}((-\infty, b]).
$$

Now for $(r,\theta, p_r, p_\theta)\in (H_a-h)^{-1}(0)\cap \{r\leq R_2\}$ we can use \eqref{b} and calculate
\begin{align*}
4ac(1 & -\varphi) +8(cr^2\varphi+1) + r^2\left(2r^2\varphi + a(1-\varphi)\right)^2 =\\
& = (1-\varphi)\left(4ac+8+a^2r^2\right)+4\varphi\left(r^6+ 2r^2+2\right)-r^2\varphi(1-\varphi)\left(a-2r^2 \right)^2\\
& \leq (1-\varphi)\left(a^2 r^2+2b\right)+\varphi\left(\sqrt{2b+ a^2 r^2}-r\left(a-2r^2 \right) \right)^2-r^2\varphi(1-\varphi)\left(a-2r^2 \right)^2\\
& = 2b+ a^2 r^2+\varphi^2 r^2 \left(a-2r^2 \right)^2-2\varphi r(a-2r^2) \sqrt{2b+ a^2 r^2}\\
& = \left( \sqrt{2b+ a^2 r^2} - r\varphi (a-2r^2)\right)^2
\end{align*}
Consequently, for $(r,\theta, p_r, p_\theta)\in (H_a-h)^{-1}(0)\cap \{r\leq R_2\}$ we have
\begin{align*}
& \sqrt{p_r^2 +\left(\frac{p_\theta}{r} + ar\right)^2} \leq \sqrt{ p_r^2+\left(\frac{p_r}{r}+r\left(2r^2\varphi + a(1-\varphi)\right) \right)^2}+r\varphi\left(a-2r^2\right)\\
& = \sqrt{ r^2\left(2r^2\varphi + a(1-\varphi)\right)^2+ 8(cr^2\varphi+1) +4ac(1-\varphi)}+r\varphi\left(a-2r^2\right)\\
& = \left|\sqrt{2b+ a^2 r^2} - r\varphi (a-2r^2)\right|+r\varphi\left(a-2r^2\right)=\sqrt{2b+ a^2 r^2},
\end{align*}
which concludes the proof.
\end{proof}

\begin{lem}\label{lem:empty}
Fix $c, d, R_1, R_2>0$ with $R_2>R_1$ and
let $H_a$ be as in \eqref{DefHa} with $a=2R_2^2$. Choose $b\geq 2ac+4$ as in Lemma \ref{lem:alpha}. Let $h$ be as in \eqref{DefHc}. Then
$$
(H_a-h)^{-1}(0)\cap \{ \{H_a-h, r\}=0\}\cap \{ \{H_a-h, \{H_a-h, r\}\}\leq 0\}\cap \{2c\geq p_\theta\} = \emptyset.
$$
\end{lem}

\begin{proof}

Using \eqref{DefHa} and \eqref{DefHc} we can calculate
\begin{align}
H_a-h & = \frac{1}{2}p_r^2+ \frac{p_\theta^2}{2r^2}+ap_\theta+\varphi \left(a -2r^2\right)(2c-p_\theta)-2ac -4\label{Ha-hc}\\
& =\frac{1}{2}p_r^2+ \frac{p_\theta^2}{2r^2}+\left(2r^2\varphi + a(1-\varphi)\right)p_\theta-4cr^2\varphi -2ac(1-\varphi)-4\nonumber\\
\{H_a-h, r\} & = d(H_a-h)(\partial_{p_r}) = p_r +\left((2r^2-a)p_\theta-4cr^2 +2ac\right) d\varphi(\partial_{p_r})\nonumber\\
\{H_a-h, p_r\} & = -d(H_a-h)(\partial_{r}) = \frac{p_\theta^2}{r^3}+4 r\varphi(2c-p_\theta) -\left(a -2r^2\right)(2c-p_\theta) d\varphi(\partial_r)\nonumber
\end{align}
By definition $d\varphi(\partial_{p_r})=\chi_0d\chi_1(\partial_{p_r})$ and $H_a^{-1}((-\infty, b])\cap \supp d\chi_1 = \emptyset$. Moreover, by Lemma \ref{lem:inclusion} we have $(H_a-h)^{-1}(0)\subseteq H_a^{-1}((-\infty, b])$. 
Hence on\linebreak $(H_a-h)^{-1}(0)$ we have 
$$
\{H_a-h, r\}=p_r\quad\textrm{and}\quad\{H_a-h, \{H_a-h, r\}\}= \{H_a-h, p_r\}.
$$
Consequently, $\{\{H_a-h, r\}=0\}\cap (H_a-h)^{-1}(0)\subseteq \{p_r=0\}$. Using \eqref{Ha-hc} we can calculate:
$$
(H_a-h)(r,\theta, 0, 0) =-4r^2\varphi-2ac(1-\varphi)  -4<0.
$$
Hence $\{p_r=p_\theta=0\}\cap (H_a-h)^{-1}(0)=\emptyset$. Consequently, we have
$$
(H_a-h)^{-1}(0) \cap \{\{H_a-h, r\}=0\} \cap \{ \{H_a-h, \{H_a-h, r\}\}\leq 0\}\cap \{p_\theta \leq 2c \}\subseteq d\varphi .
$$
However on $H_a^{-1}((-\infty, b])$ we have
$-d\varphi(\partial_r)=-d\chi_0(\partial_r) = - \frac{\chi'}{R_2-R_1} \geq 0$.
Therefore, 
$$
\{H_a-h, p_r\}> 0\quad\textrm{on}\quad  H_a^{-1}((-\infty, b])\cap \{p_\theta\leq 2c \}\cap \{p_\theta \neq 0 \}.
$$
We can conclude that $\{H_a-h, p_r\}> 0$ on the set
$$
 (H_a-h)^{-1}(0)\cap \{\{H_a-h, r\}=0\}\cap \{ p_\theta \leq 2c\}.
$$
\end{proof}

The following proof follows closely the methodology introduced in \cite{Wisniewska2025}:

\vspace*{.5cm}
\noindent \textit{Proof of Theorem \ref{thm:Kepler}:}

Fix $q_0, q_1 \in \R^2$ and set $R_1:=\max\{|q_0|,|q_1|\}$. Choose $c>\sqrt{2}R_1$ and $R_2 > R_1$. Let $H_a$ be as in \eqref{DefHa} with $a:=2R_2^2$. Choose constants $b>0$ and $d>0$ as in Lemma \ref{lem:alpha} and Lemma \ref{lem:ContactType}, respectively. With those constants define $h$ as in \eqref{DefHc}. By Lemma \ref{lem:inclusion} and Lemma \ref{lem:ContactType} we know that $h$ is an element of the set $\mathcal{H}$ as defined in \eqref{DefH}. We will show that $h$ satisfies assertions of Theorem \ref{thm:Kepler}.

Denote by $B(R_1)$ a ball in $\R^2$ or radius $R_1$ and center in the origin.
We will first show that
$$
\left\lbrace (v,\eta)\in \Crit\A^{H_a-h}_{q_0,q_1}\ \big|\ v([0,1])\subseteq T^*B(R_1)\right\rbrace \subseteq \Crit\A^{K_c}_{q_0,q_1}.
$$
Later we will prove that
$$
\left\lbrace (v,\eta)\in \Crit\A^{H_a-h}_{q_0,q_1}\ \big|\ v([0,1])\subseteq T^*B(R_1)\right\rbrace = \Crit\A^{H_a-h}_{q_0,q_1}.
$$

Consider $(v,\eta)\in \Crit\A^{H_a-h}_{q_0,q_1}$, such that $v([0,1])\subseteq T^*B(R_1)$. By assumption $v([0,1])\subseteq T^*B(R_1) \cap (H_a-h)^{-1}(0)$. By Lemma \ref{lem:inclusion} we know that\linebreak $(H_a-h)^{-1}(0) \subseteq H_a^{-1}((-\infty, b])$. Hence
$$
v([0,1])\subseteq T^*B(R_1) \cap H_a^{-1}((-\infty, b])= \varphi^{-1}(1),
$$
for the function $\varphi$ defined in \eqref{DefVarphi}. By construction of $h$ as in \eqref{DefHc}\linebreak the Hamiltonians $K_c$ as in \eqref{regHam} and $H_a-h$ agree on $\varphi^{-1}(1)$. Consequently  $(v,\eta)\in \Crit\A^{K_c}_{q_0,q_1}$, which proves the first inclusion.

We will now prove the second inclusion by contradiction. Suppose $(v,\eta)\in  \Crit\A^{H_a-h}_{q_0,q_1}$, such that $v([0,1])\setminus T^*B(R_1)\neq \emptyset$. Denote $v(t)=(r(t), \theta(t), p_r(t), p_\theta(t))$ in the spherical coordinates on $\R^2$. Then there exists $t_0\in (0,1)$, such that $r(t_0)=\max_{t\in [0,1]}r(t)>R_1$. For such $t_0$ we would have
\begin{align*}
r'(t_0) & =\{H_a-h, r\}\circ v(t_0)=0,\\
r''(t_0) & = \{H_a-h,\{H_a-h, r\}\}\circ v (t_0)\leq 0.
\end{align*}
In other words, we would have
$$
v(t_0) \in (H_a-h)^{-1}(0)\cap \{ \{H_a-h, r\}=0\}\cap \{ \{H_a-h, \{H_a-h, r\}\}\leq 0\}.
$$

Now observe that by construction $H_a-h$ does not depend on $\theta$, thus\linebreak $\{H_a-h,p_\theta\}=0$ and consequently we have $p_\theta (t_0)=p_\theta(0)=p_\theta(1)$. Since $R_1 = \max\{|q_0|,|q_1|\}$ therefore by Lemma \ref{lem:inclusion} we have
$$
v(0), v(1)\in T^*B(R_1)\cap (H_a-1)^{-1}(0)\subseteq T^*B(R_1) \cap H_a^{-1}((-\infty, b])= \varphi^{-1}(1).
$$
As the two Hamiltonians $H_a-h$ and $K_c$ agree on $\varphi^{-1}(1)$, we have $v(0), v(1)\in K_c^{-1}(0)$.
Now since we have chosen $c> \sqrt{2}R_1 = \sqrt{2}\max\{|q_0|,|q_1|\}$, then by Lemma \ref{lem:pTheta} we know that $p_\theta(0)=p_\theta(1)\leq 2c$. Consequently,
$$
v(t_0) \in (H_a-h)^{-1}(0)\cap \{ \{H_a-h, r\}=0\}\cap \{ \{H_a-h, \{H_a-h, r\}\}\leq 0\}\cap \{p_\theta \leq 2c\}.
$$
But by Lemma \ref{lem:empty} the set on the right hand side is empty, which brings us to a contradiction.

\hfill $\square$

\printbibliography
\end{document}